\documentclass[11pt]{amsart}
\usepackage{latexsym,amssymb,amsfonts,amsmath,graphicx,fullpage,url,color,tikz}
\usepackage[utf8]{inputenc}
\usepackage[vcentermath,enableskew]{youngtab}
\usepackage[all]{xy}
\usepackage{verbatim}
\usepackage{ytableau}
\usepackage{makecell}
\usepackage[export]{adjustbox}
\usepackage{multicol}

\newtheorem{theorem}{Theorem}[section]
\newtheorem{proposition}[theorem]{Proposition}
\newtheorem{lemma}[theorem]{Lemma}
\newtheorem{corollary}[theorem]{Corollary}
\newtheorem*{MainTheorem1}{Theorem~\ref{thm:pasmverts}}
\newtheorem*{MainTheorem2}{Theorem~\ref{pasmfacets}}
\newtheorem*{MainTheorem3}{Theorem~\ref{thm:permuto_ineq}}
\newtheorem*{MainTheorem4}{Theorem~\ref{permutofacets}}
\newtheorem*{MainTheorem5}{Theorem~\ref{chain-face-lattice}}
\newtheorem*{MainTheorem6}{Theorem~\ref{pasm-proj}}
\newtheorem*{PasmFaceLatticeThm}{Theorem~\ref{thm:pasm_facelattice}}
\newtheorem*{Volthm}{Theorem~\ref{ppermvol1}}

\newtheorem{conjecture}[theorem]{Conjecture}

\newtheorem*{Faceconj}{Conjecture~\ref{mn-chain-face}}
\newtheorem*{Volconj2}{Conjecture~\ref{ppermvol2}}
\theoremstyle{definition}
\newtheorem{definition}[theorem]{Definition}
\newtheorem{example}[theorem]{Example}

\theoremstyle{remark}
\newtheorem{remark}[theorem]{Remark}
\numberwithin{equation}{section}
\usepackage{multicol}

\newcommand{\R}{\mathcal{R}}
\newcommand{\Q}{\mathcal{Q}}
\newcommand{\eR}{\widetilde{\mathcal{R}}}

\newcommand{\pasm}{\mathrm{PASM}}
\newcommand{\pperm}{\mathrm{PPerm}}

\author{Dylan Heuer and Jessica Striker}
\email{heuerd@msoe.edu, jessica.striker@ndsu.edu}

\address{Milwaukee School of Engineering, North Dakota State University}
\title{Partial Permutation and Alternating Sign Matrix Polytopes}

\begin{document}
\begin{abstract}
We define and study a new family of polytopes which are formed as convex hulls of partial alternating sign matrices. We determine the inequality descriptions, number of facets, and face lattices of these polytopes. We also study partial permutohedra that we show arise naturally as projections of these polytopes. We enumerate facets and also characterize the face lattices of partial permutohedra in terms of chains in the Boolean lattice. Finally, we have a result and a conjecture on the volume of partial permutohedra when one parameter is fixed to be two.
\end{abstract}

\maketitle
\tableofcontents

\section{Introduction}
Many examples of polytopes are either \emph{simple} (every vertex is contained in the minimal number of facets), such as the $n$-cube, or \emph{simplicial} (every proper face is a simplex), such as the tetrahedron.
An interesting example of a non-simple and non-simplicial polytope is the $n$th Birkhoff polytope (for $n > 3$), defined as the convex hull of $n\times n$ permutation matrices~\cite{birkhoff, vonneumann}. This polytope has $n^2$ facets, its vertices are exactly the $n\times n$ permutation matrices, and it is easily described via inequalities as the set of all doubly stochastic matrices \cite{birkhoff, vonneumann}. It also has a characterization of its face lattice in terms of elementary bipartite graphs \cite{BrualdiGibson1976,BilleraSarangarajan1996}. Another non-simple and non-simplicial example is the $n$th alternating sign matrix polytope (for $n\geq 3$), defined as the convex hull of $n\times n$ alternating sign matrices. This polytope has $4[(n-2)^2+1]$ facets, and its vertices are exactly the $n \times n$ alternating sign matrices, and it has a nice inequality description \cite{behrend, striker}. Its face lattice can be characterized in terms of elementary flow grids \cite{striker}. Alternating sign matrices are interesting mathematical objects in their own right. The poset of alternating sign matrices is the MacNeille completion of the Bruhat order on permutation matrices \cite{TREILLIS}, and alternating sign matrices are in bijection with many interesting objects (see, for example \cite{ProppManyFaces}). 

In this paper, we define and study a more general class of polytopes, denoted $\pasm(m,n)$, composed as convex hulls of $m\times n$ \emph{partial alternating sign matrices} (see Definition~\ref{pasm}). These matrices are also of independent interest; there are results about posets \cite{Fortin} and bijections \cite{Heuer} analogous to those in the non-partial case. This paper continues the study of analogous results in the realm of polytopes and reveals new connections to graph associahedra. We use machinery developed in this study of sign matrix polytopes~\cite{SolhjemStriker} to determine the inequality descriptions of these polytopes, as well as facet enumerations and a description of their face lattices. Finally, we investigate the \emph{partial permutohedron} $\mathcal{P}(m,n)$ (see Definition~\ref{def:partialpermutohedron}), and show it is a projection of the polytopes from the first part of the paper.

Below we state our main results.
Our first set of main results involves the partial alternating sign matrix polytope $\pasm(m,n)$, while our second concerns the partial permutohedron $\mathcal{P}(m,n)$. For each set, we find an inequality description, enumerate the facets, and characterize the face lattice.

\begin{MainTheorem1}
The polytope $\pasm(m,n)$ consists of all $m\times n$ real matrices $X = (X_{ij})$ such that:
\begin{align*}
0 \leq \displaystyle\sum_{i'=1}^{i} X_{i'j} & \leq 1, & \mbox{ for all } & 1 \leq i \leq m, 1 \leq j \leq n, \\
0 \leq \displaystyle\sum_{j'=1}^{j} X_{ij'} & \leq 1, & \mbox{ for all } & 1 \leq i \leq m, 1 \leq j \leq n.
\end{align*}
\end{MainTheorem1}

\begin{MainTheorem2}
The number of facets of $\pasm(m,n)$ equals $4mn - 3m - 3n + 5$.
\end{MainTheorem2}

See Definitions \ref{sum-labeling} and \ref{def:union} through \ref{def:region} for the relevant notation and terminology in the following theorem.
\begin{PasmFaceLatticeThm}
Let $F$ be a face of $\pasm(m,n)$ and $\mathcal{M}(F)$ be the set of partial alternating sign matrices that are vertices of $F$. The map $\psi:F\mapsto g(\mathcal{M}(F))$ induces an isomorphism between the face lattice of $\pasm(m,n)$ and the set of sum-labelings of ${\Gamma}_{(m,n)}$ ordered by containment. Moreover, the dimension of $F$ equals the number of regions of $\psi(F)$.
\end{PasmFaceLatticeThm}

The following three theorems from Section~\ref{sec:ppermuto} comprise our second set of main results.
 \begin{MainTheorem3}
The polytope $\mathcal{P}(m,n)$ consists of all vectors $u\in\mathbb{R}^m$  such that:
\begin{align*}
\displaystyle\sum_{i \in S} u_{i} & \leq \binom{n+1}{2}-\binom{n-k+1}{2}, & \mbox{ where } S \subseteq \{1,\ldots,m\}, |S| = k \neq 0 , \mbox{ and } \\
u_i & \geq 0, & \mbox{ for all } 1 \leq i \leq m.
\end{align*}
\end{MainTheorem3}

\begin{MainTheorem4} 
The number of facets of $\mathcal{P}(m,n)$ equals $m + 2^m -1 - \displaystyle\sum_{r=1}^{m-n}\binom{m}{m-r}$.
\end{MainTheorem4}
 
We reinterpret the result \cite[Prop.\ 56]{Manneville-Pilaud} that $\mathcal{P}(m,m)$ is a graph associahedron called the \emph{stellohedron} to prove an alternate characterization of its face lattice in terms of chains in the Boolean lattice $\mathcal{B}_m$. See Definition \ref{def:missing} for the notion of missing ranks.

\begin{MainTheorem5}
The face lattice of $\mathcal{P}(m,m)$ is isomorphic to the lattice of chains in $\mathcal{B}_m$, where $C < C'$ if $C'$ can be obtained from $C$ by iterations of (1) and/or (2) from Lemma~\ref{chain-contain}. A face of $\mathcal{P}(m,m)$ is of dimension $k$ if and only if the corresponding chain has $k$ missing ranks.
\end{MainTheorem5}

We conjecture a similar face lattice characterization for $\mathcal{P}(m,n)$ in the case $m\neq n$.

\begin{Faceconj}
Faces of $\mathcal{P}(m,n)$ are in bijection with chains in $\mathcal{B}_m$ whose difference between largest and smallest nonempty subsets is at most $n-1$. A face of $\mathcal{P}(m,n)$ is of dimension $k$ if and only if the corresponding chain has $k$ missing ranks.
\end{Faceconj}

We furthermore connect these polytopes by showing that $\pasm(m,n)$ projects to $\mathcal{P}(m,n)$, by a similar technique used to show that alternating sign matrix polytopes project to permutohedra~\cite{striker}. Our last main result is as follows; here $\phi_z$ is the map that multiplies a matrix by $z$ on the right and $\mathcal{P}_z$ is a generalized partial permutohedron determined by $z$.

\begin{MainTheorem6}
Let $z$ be a strictly decreasing vector in $\mathbb{R}^m$. Then
$\phi_z(\pasm(m,n))=\mathcal{P}_z(n,m)$. 
\end{MainTheorem6}

This projection connects matrix polytopes to graph associahedra. We also explore connections to chains in the Boolean lattice. These connections are helpful conceptually, as they relate the face structure of these polytopes to familiar combinatorial objects. 

Finally, we have computed the normalized volume and Ehrhart polynomials of the polytopes studied in this paper. We note that the Ehrhart polynomials we were able to compute have positive coefficients, and have found the following result and conjecture regarding the normalized volume of the partial permutohedron.

\begin{Volthm}
The polytope $\mathcal{P}(2,n)$ has normalized volume equal to $2n^2-1$.
\end{Volthm}
\begin{Volconj2}
The polytope $\mathcal{P}(m,2)$ has normalized volume equal to $3^m-m$.
\end{Volconj2}

Our outline is as follows.
In Section~\ref{sec:PermMatrices}, we introduce definitions and notation for the families of matrices used to create our polytopes. In Section~\ref{sec:PPermPoly}, we summarize known results on partial permutation polytopes. In Section~\ref{sec:pasm}, we define partial alternating sign matrix polytopes and determine their inequality descriptions, facet enumerations, and face lattice description. In Section~\ref{sec:ppermuto}, we define partial permutohedra. We then determine inequality descriptions and facet enumerations and characterize the face lattices using chains in the Boolean lattice.
We show that the polytopes from Sections~\ref{sec:PPermPoly} and \ref{sec:pasm} project to these partial permutohedra. Finally, at the end of each of Sections~\ref{sec:PPermPoly}, \ref{sec:pasm}, and \ref{sec:ppermuto}, we discuss volumes.

\section{Matrices}
\label{sec:PermMatrices}

In this section, we discuss matrices which generalize permutation matrices and alternating sign matrices. Then in the next section, we study their corresponding polytopes. 
\subsection{Partial permutation matrices}
\label{subsec:pperm}
We begin with the following definition.
\begin{definition}
\label{pperm}
An $m \times n$ \emph{partial permutation matrix} is an $m \times n$ matrix $M = \left(M_{ij}\right)$ with entries in $\{0,1\}$ such that:
\begin{align}
\label{eq:pperm1}
\displaystyle\sum_{j'=1}^{n} M_{ij'} &\in \left\{0,1\right\}, & \mbox{ for all } 1 \leq i \leq m. \\
\label{eq:pperm2}
\displaystyle\sum_{i'=1}^{m} M_{i'j} &\in \left\{0,1\right\}, & \mbox{ for all } 1 \leq j \leq n.
\end{align}
We denote the set of all $m \times n$ partial permutation matrices $P_{m,n}$.
\end{definition}

\begin{remark}
Partial permutations matrices are sometimes called \emph{subpermutation matrices}, see, for example, \cite{Brualdi-Ryser}. We choose the terminology {partial permutation} since we consider our matrices as objects in their own right, rather than as submatrices of larger (square) permutation matrices. The use of the term partial permutation is consistent with literature on square partial permutation matrices, such as \cite{Cao-et-all}. Rectangular partial permutation matrices are mentioned in \cite{Ouchterlony}.
\end{remark}


\begin{remark}
\label{prop:ppermcard}
For any $m$ and $n$, the set of partial permutation matrices $P_{m,n}$ is enumerated by:
\begin{equation}
\displaystyle\sum_{k=0}^{\min(m,n)} \binom{m}{k}\binom{n}{k}k!.
\end{equation}
This follows using standard counting arguments. We first we choose which of the $m$ rows will have a $1$ in them, in $\binom{m}{k}$ ways. For each of these rows, we determine in what column that $1$ will be. There are $n(n-1)\cdots(n-k+1)=\binom{n}{k}k!$ ways to do this.
\end{remark}

\begin{example}
\label{ex:pperms}
The $13$ elements of $P_{2,3}$ are:
\begin{multicols}{3}
$M_1 = \begin{pmatrix} 0 & 0 & 0 \\ 0 & 0 & 0 \end{pmatrix}$

\vspace{0.25cm}

$M_2 = \begin{pmatrix} 1 & 0 & 0 \\ 0 & 0 & 0 \end{pmatrix}$

\vspace{0.25cm}

$M_3 = \begin{pmatrix} 0 & 1 & 0 \\ 0 & 0 & 0 \end{pmatrix}$

\vspace{0.25cm}

$M_4 = \begin{pmatrix} 0 & 0 & 1 \\ 0 & 0 & 0 \end{pmatrix}$

\vspace{0.25cm}

$M_5 = \begin{pmatrix} 0 & 0 & 0 \\ 1 & 0 & 0 \end{pmatrix}$

\vspace{0.25cm}

$M_6 = \begin{pmatrix} 0 & 0 & 0 \\ 0 & 1 & 0 \end{pmatrix}$

\vspace{0.25cm}

$M_7 = \begin{pmatrix} 0 & 0 & 0 \\ 0 & 0 & 1 \end{pmatrix}$

\vspace{0.25cm}

$M_8 = \begin{pmatrix} 1 & 0 & 0 \\ 0 & 1 & 0 \end{pmatrix}$

\vspace{0.25cm}

$M_9 = \begin{pmatrix} 1 & 0 & 0 \\ 0 & 0 & 1 \end{pmatrix}$

\vspace{0.25cm}

$M_{10} = \begin{pmatrix} 0 & 1 & 0 \\ 1 & 0 & 0 \end{pmatrix}$

\vspace{0.25cm}

$M_{11} = \begin{pmatrix} 0 & 1 & 0 \\ 0 & 0 & 1 \end{pmatrix}$

\vspace{0.25cm}

$M_{12} = \begin{pmatrix} 0 & 0 & 1 \\ 1 & 0 & 0 \end{pmatrix}$

\vspace{0.25cm}

$M_{13} = \begin{pmatrix} 0 & 0 & 1 \\ 0 & 1 & 0 \end{pmatrix}$

\end{multicols}
\end{example}

\subsection{Partial alternating sign matrices}
\label{subsec:pasm}
In this subsection, we consider a superset of partial permutation matrices.

\begin{definition}
\label{pasm}
An $m \times n$ \emph{partial alternating sign matrix} is an $m \times n$ matrix $M = \left(M_{ij}\right)$ with entries in $\left\{-1,0,1\right\}$ such that:
\begin{align}
\label{eq:pasm1}
\displaystyle\sum_{i'=1}^{i} M_{i'j} &\in \left\{0,1\right\}, & \mbox{ for all } 1 \leq i \leq m , 1 \leq j \leq n. \\
\label{eq:pasm2}
\displaystyle\sum_{j'=1}^{j} M_{ij'} &\in \left\{0,1\right\}, & \mbox{ for all } 1 \leq i \leq m, 1 \leq j \leq n.
\end{align}
We denote the set of all $m \times n$ partial alternating sign matrices as $\pasm_{m,n}$.
\end{definition}

\begin{remark}
\label{remark:mx_contain}
The set of matrices in $\pasm_{m,n}$ with no $-1$ entries is the set $P_{m,n}$ of partial permutation matrices.
\end{remark}

\begin{example}
\label{ex:pasms}
The set $\pasm_{2,3}$ consists of the 13 matrices from Example~\ref{ex:pperms} plus the following four additional matrices:
\begin{multicols}{4}
$M_{14} = \begin{pmatrix} 0 & 1 & 0 \\ 1 & -1 & 0 \end{pmatrix}$

$M_{15} = \begin{pmatrix} 0 & 1 & 0 \\ 1 & -1 & 1 \end{pmatrix}$

$M_{16} = \begin{pmatrix} 0 & 0 & 1 \\ 1 & 0 & -1 \end{pmatrix}$

$M_{17} = \begin{pmatrix} 0 & 0 & 1 \\ 0 & 1 & -1 \end{pmatrix}$
\end{multicols}
\end{example}

\begin{remark}\label{remark:signmatrices}
Partial alternating sign matrices are a subset of \emph{sign matrices}, which differ from Definition~\ref{pasm} in that each row partial sum is not restricted to $\{0,1\}$ as in (\ref{eq:pasm2}), but may equal any non-negative integer. See \cite{SolhjemStriker} for information about polytopes whose vertices are sign matrices and Lemma~\ref{lem:PmnPASMmn} for the relationship between these polytopes and polytopes  whose vertices are partial alternating sign matrices (see Definition~\ref{pasmpoly}).
\end{remark}

The cardinality of $\pasm_{n,n}$ is given by OEIS sequence A202751 \cite{oeis1}. It is unlikely that there exists a product formula for $|\pasm_{m,n}|$, since, for example, $|\pasm_{6,6}|=1442764=2^2 \cdot 373 \cdot 967$. Note that $|\pasm_{m,n}| = |\pasm_{n,m}|$, since an $n \times m$ partial alternating sign matrix is the transpose of an $m \times n$ partial alternating sign matrix. These cardinalities are given by OEIS sequence A202756 \cite{oeis1}. The bijection between partial alternating sign matrices and the objects described for these sequences in the OEIS is given by an analog of the corner-sum map in the usual alternating sign matrix case (see, for example, \cite{Heuer}).

\begin{remark} The set of $n \times n$ partial alternating sign matrices were studied in a different context 
by Fortin~\cite{Fortin}. He showed that, with a certain poset structure, the lattice of partial alternating sign matrices is the MacNeille completion of the poset of partial permutations (which he called partial injective functions). This is analogous to the result of Lascoux and Sch\"{u}tzenberger
~\cite{TREILLIS} that the lattice of $n\times n$ alternating sign matrices is the MacNeille completion of the strong Bruhat order on $S_n$. Partial alternating sign matrices were also defined in \cite{Osculating}, and studied there in the context of the enumeration of certain osculating lattice paths.
\end{remark}

\section{Partial permutation polytopes}
\label{sec:PPermPoly}
In this section, we give the definition of partial permutation polytopes and review their inequality descriptions and the enumeration of their vertices and facets. These results are known (see Remark~\ref{disclaimer}) or easily deduced, but we include them for completeness and for comparison to the polytopes in the next section. We also compute data for the volume of these polytopes and give a formula for $m=2$ in Theorem~\ref{conj:volPPerm}.

\begin{definition}
\label{ppermpoly}
Let $\pperm(m,n)$ be the polytope defined as the convex hull, as vectors in $\mathbb{R}^{mn}$, of all the matrices in $P_{m,n}$.
Call this the \emph{(m,n)-partial permutation polytope}.
\end{definition}

\begin{remark}
\label{pperm_dim}
The dimension of $\pperm(m,n)$ is $mn$. To see this, let $U_{i,j}$ be the $m \times n$ matrix with $(i,j)$ entry equal to 1 and zeros elsewhere. Note that $U_{i,j} \in \pperm(m,n)$ for all $1 \leq i \leq m$, $1 \leq j \leq n$. Since $\pperm(m,n)$ contains each of these $mn$ unit vectors, its dimension equals the ambient dimension $mn$. 
\end{remark}

\begin{remark}
\label{disclaimer}
The polytopes $\pperm(m,n)$ have been previously studied in different contexts. Since any partial permutation matrix can be reinterpreted as an incidence vector of some matching, $\pperm(m,n)$ is a \emph{matching polytope}. In \cite{BalinskiRussakoff,Chvatal}, adjacency conditions of vertices of matching polytopes were studied. For a nice summary and proof of these results, see \cite[Chapter 25]{Schrijver}. Also, Mirsky showed that the set of $n \times n$ doubly substochastic matrices is the convex hull of all $n \times n$ partial permutation matrices \cite{Mirsky_sub}. This is easily extendable to the $m \times n$ case, which is stated below in Proposition~\ref{ppermcircuit}. As such, $\pperm(m,n)$ is often referred to in the literature as the polytope of $m \times n$ doubly substochastic matrices \cite[Section 9.8]{BrualdiBook2006}. In \cite{AaronAllen}, partial permutations were viewed as rook placements, and edges and faces of $\pperm(m,n)$ were enumerated.
\end{remark}

\begin{proposition}
\label{ppermcircuit}
The polytope $\pperm(m,n)$ consists of all $m \times n$ real matrices $X = (X_{ij})$ such that:
\begin{align}
\label{eq:ppermpoly1}
X_{ij} & \geq 0 , & \mbox{ for all } & 1 \leq i \leq m, 1 \leq j \leq n, \\ 
\label{eq:ppermpoly2}
\displaystyle\sum_{j'=1}^n X_{ij'} & \leq 1, & \mbox{ for all } & 1 \leq i \leq m, \\ 
\label{eq:ppermpoly3}
\displaystyle\sum_{i'=1}^n X_{i'j} & \leq 1, & \mbox{ for all } & 1 \leq j \leq n.
\end{align}
\end{proposition}

Since these inequalities are irredundant, we obtain the following as a corollary.
\begin{corollary}
The number of facets of $\pperm(m,n)$ equals $mn+m+n$. 
\end{corollary}

We may also easily count the vertices using the fact that this is a $0$-$1$ polytope, which implies that each matrix in $P_{m,n}$ is extreme. See also \cite[Theorem 9.8.3]{BrualdiBook2006}.
\begin{proposition}
The vertices of $\pperm(m,n)$ are exactly the matrices in $P_{m,n}$, so $\pperm(m,n)$ has $\displaystyle\sum_{k=0}^{\min(m,n)} \binom{m}{k}\binom{n}{k}k!$ vertices.
\end{proposition}

\begin{remark}
The normalized volume of $\pperm(m,n)$ for small values of $m$ and $n$ is given in Figure~\ref{pperm_vol}, computed using SageMath~\cite{sage}. These computations for $\pperm(2,2)$ and $\pperm(3,3)$ appear in \cite[p.\ 609]{Stanley1973}. In \cite[Table 2]{KohlOlsenSanyal}, the normalized volume for $\pperm(4,4)$ appears, and lower bound is obtained for $\pperm(5,5)$.
\end{remark}

\begin{figure}[hbtp]
\begin{tabular}{|c|c|c|c|c|c|}
\hline
\diaghead(1,-1)%
   {\theadfont nnn}%
   {$m$}{$n$} & 1 & 2   & 3      & 4          & 5              \\ \hline
    1     & 1 & 1   & 1      & 1          &  1             \\ \hline
    2     & 1 & 4   & 17     & 66         & 247            \\ \hline
    3     & 1 & 17  & 642    & 22148      & 622791         \\ \hline
    4     & 1 & 66  & 22148  & 12065248   & 5089403019     \\ \hline
    5     & 1 & 247 & 622791 & 5089403019 & 53480547965190 \\ \hline
\end{tabular}
\caption{The normalized volume of $\pperm(m,n)$ for small values of $m$ and $n$.}
\label{pperm_vol}
\end{figure}

 Note that there does not appear to be a nice general formula for these volumes. However, when one parameter is set equal to two, we conjectured the following in an earlier draft of this paper. We thank the anonymous referee for the proof.

\begin{theorem}
\label{conj:volPPerm}
The normalized volume of $\pperm(n,2)$ (or equivalently $\pperm(2,n)$) is equal to $\binom{2n}{n}-n$.
\end{theorem}

\begin{proof}
The result will be proved by obtaining the Ehrhart polynomial of $\pperm(2,n)$, as follows.

Consider any positive integer $n$ and nonnegative integer $r$, and let $\R(n,r)$ be the
set of $2\times n$ matrices with nonnegative integer entries such that the sum of entries in each
row is at most $r$. Since the number of $n$-tuples of nonnegative integers with sum at most
$r$ is $\binom{n+r}{n}$, it follows that
\begin{equation}
\label{conj1}
|\R(n,y)|=\binom{n+r}{n}^2.
\end{equation}
For any matrix in $\R(n,r)$, the sum of all entries is at most $2r$, and hence there can be no
more than one column whose sum of entries is greater than $r$. Let $\eR(n, r)$ be the set of
matrices in $\R(n, r)$ for which the sum of entries in each column is at most $r$, and $\R(n, r)_j$
be the set of matrices in $\R(n, r)$ for which the sum of entries in column $j$ is greater than $r$.
It can be seen, by interchanging matrix columns, that $\R(n, r)_j$ is independent of $j$, for
$j = 1,\ldots,n$. Since $\R(n, r)$ is the disjoint union of $\eR(n, r), \R(n, r)_1,\ldots,\R(n, r)_n$, it follows that
\begin{equation}
\label{conj2}
|\R(n,r)|=|\eR(n,r)|+n|\R(n,r)_n|.
\end{equation}
Now let $\Q(n, r)$ be the set of $2n$-tuples of nonnegative integers with sum at most $r - 1$.
Then
\begin{equation}
\label{conj3}
|\Q(n,r)|=\binom{2n+r-1}{2n}.
\end{equation}
It can easily be shown that a bijection from $\R(n, r)_n$ to $Q(n, r)$ is obtained by mapping any 
$\begin{pmatrix}
x_{1,1} & \cdots & x_{1,n} \\
x_{2,1} & \cdots & x_{2,n}
\end{pmatrix}\in\R(n, r)_n$ to
$(x_{1,1},\ldots,x_{1,n-1},r-\sum_{j=1}^{n}x_{1,j},x_{2,1},\ldots,x_{2,n-1},r-\sum_{j=1}^{n}x_{2,j})$. Note that the entries of this $2n$-tuple are nonnegative integers with sum $2r - x_{1,n} - x_{2,n}$, which is at most $r-1$ since $x_{1,n}+x_{2,n}$ is greater than $r$. Note also that the inverse bijection is obtained by mapping any $(y_1,\ldots,y_{2n})\in\Q(n,r)$ to 
\[\begin{pmatrix}
y_{1} & \cdots & y_{n-1} & r-\sum_{i=1}^{n} y_i \\
y_{n+1} & \cdots & y_{2n-1} & r-\sum_{i=n+1}^{2n} y_i
\end{pmatrix}.\]
Therefore $|\R(n,r)_n)|=|\Q(n,r)|$, and using this together with (\ref{conj1})--(\ref{conj3}) gives 
\begin{equation}
\label{conj4}
|\eR(n,r)|=\binom{n+r}{n}^2-n\binom{2n+r-1}{2n}.
\end{equation}
Finally, observe that $\eR(n, r)$ is the set of integer points of the $r$-th dilate of the integral
polytope $\pperm(2, n)$, and hence $|\eR(n, r)|$ as a function of $r$ is the Ehrhart polynomial
of $\pperm(2, n)$. The normalized volume of any integral $d$-dimensional polytope $P$ in $\mathbb{R}^d$
is $d!$ times the coefficient of $r^d$ in the Ehrhart polynomial of $P$ as a function of $r$. Since
$\pperm(2, n)$ is a $2n$-dimensional polytope in $\mathbb{R}^{2n}$, and the coefficient of $r^{2n}$ on the right-hand side of (\ref{conj4}) is $1/(n!)^2-n/(2n)!$, it follows that the normalized volume of $\pperm(2, n)$ is $\binom{2n}{n}-n$, as required.
\end{proof}

\begin{remark}
We have used SageMath to compute the Ehrhart polynomials for $\pperm(m,n)$ for $m,n \leq 5$ and note that in all of these cases their coefficients are positive.
\end{remark}

\section{Partial alternating sign matrix polytopes}
\label{sec:pasm}
In this section, we define partial alternating sign matrix polytopes. We give an inequality description and facet enumeration in Subsection~\ref{subsec:pasmVFI}. In Subsection~\ref{subsec:pasmfacelattice}, we determine the face lattice. We also compute the volume for small values of $m$ and $n$ in Subsection~\ref{subsec:pasmvolume}.

\subsection{Vertices, facets, inequality description}
\label{subsec:pasmVFI}
In this subsection, we give the definition of partial alternating sign matrix polytopes. In Proposition~\ref{prop:vertpasm}, we determine the vertices. We prove an inequality description in Theorem~\ref{thm:pasmverts}. Then in Theorem~\ref{pasmfacets}, we enumerate the facets.

\begin{definition}
\label{pasmpoly}
Let $\pasm(m,n)$ be the polytope defined as the convex hull, as vectors in $\mathbb{R}^{mn}$, of all the matrices in $\pasm_{m,n}$. 
Call this the $(m,n)$-\emph{partial alternating sign matrix polytope}.
\end{definition}

\begin{remark}
\label{pasm_dim}
$\pasm(m,n)$ contains $\pperm(m,n)$, since, as noted in Remark~\ref{remark:mx_contain}, the set of partial alternating sign matrices $\pasm_{m,n}$ contains all the partial permutation matrices $P_{n,m}$. So the dimension of $\pasm(m,n)$ is the ambient dimension $mn$, since by Remark~\ref{pperm_dim}, this is the dimension of $\pperm(m,n)$.
\end{remark}

\begin{proposition}
\label{prop:vertpasm}
The vertices of $\pasm(m,n)$ are exactly the matrices in $\pasm_{m,n}$.
\end{proposition}

\begin{proof}
All $m \times n$ sign matrices are vertices of the sign matrix polytope (defined as their convex hull) \cite[Theorem 4.3]{SolhjemStriker}. 
Since $m \times n$ partial alternating sign matrices are a subset of $m \times n$ sign matrices (see Remark~\ref{remark:signmatrices}), it is immediate that they are also vertices of their convex hull.
\end{proof}

We now give the following definitions from \cite{SolhjemStriker}, which we will use in the proof of Theorem~\ref{thm:pasmverts}.

\begin{definition}[\protect{\cite[Definition 3.3]{SolhjemStriker}}]
\label{def:gamma}
We define the $m \times n$ \emph{grid graph} $\Gamma_{(m,n)}$ as follows.
The vertex set is $V(m,n):=\{(i,j)$ : $1 \leq i \leq m+1$, $1 \leq j \leq n+1  \} - \{(m+1,n+1)\}$.
We separate the vertices into two categories. We say the \emph{internal vertices} are \{$(i,j)$ : $1 \leq i \leq m$, $1 \leq j \leq n$\} and the \emph{boundary vertices} are \{$(m+1,j) \mbox{ and } (i,n+1)$ : $1 \leq i \leq m$, $1 \leq j \leq n$\}. The edge set is:
\[E(m,n):=  \begin{cases} (i,j) \text{ to } (i+1,j) & 1 \leq i \leq m, 1 \leq j \leq n\\ (i,j) \text{ to } (i,j+1) & 1 \leq i \leq m, 1 \leq j \leq n. \end{cases}\]
Edges between internal vertices are called \emph{internal edges} and any edge between an internal and boundary vertex is called a \emph{boundary edge}. We draw the graph with $j$ increasing to the right and $i$ increasing down, to correspond with matrix indexing.
\end{definition}

\begin{definition}[\protect{\cite[Definition 3.4]{SolhjemStriker}}]
\label{def:Xhat}
Given an $m \times n$ matrix $X$, we define a labeled graph, $\hat{X}$, which is a labeling of the vertices and edges of $\Gamma_{(m,n)}$ from Definition~\ref{def:gamma}. 
The horizontal edges from $(i,j)$ to $(i,j+1)$ are each labeled by the corresponding row partial sum $r_{ij}= \displaystyle\sum_{j'=1}^{j} X_{ij'}$ ($1\leq i \leq m, 1 \leq j \leq n$). Likewise, the vertical edges from $(i,j)$ to $(i+1,j)$ are each labeled by the corresponding column partial sum $c_{ij} = \displaystyle\sum_{i'=1}^{i} X_{i'j}$ ($1\leq i \leq m, 1 \leq j \leq n$).
\end{definition}

The following theorem gives an inequality description of $\pasm(m,n)$. The proof uses a combination of ideas from \cite{SolhjemStriker, striker}.

\begin{theorem}
\label{thm:pasmverts}
The polytope $\pasm(m,n)$ consists of all $m\times n$ real matrices $X = (X_{ij})$ such that:
\begin{align}
\label{eq:pasmpoly1}
0 \leq \displaystyle\sum_{i'=1}^{i} X_{i'j} & \leq 1, & \mbox{ for all } & 1 \leq i \leq m, 1 \leq j \leq n, \\
\label{eq:pasmpoly2}
0 \leq \displaystyle\sum_{j'=1}^{j} X_{ij'} & \leq 1, & \mbox{ for all } & 1 \leq i \leq m, 1 \leq j \leq n.
\end{align}
\end{theorem}

\begin{proof}
Let $X\in\pasm(m,n)$. First we need to show that $X$ satisfies (\ref{eq:pasmpoly1}) and (\ref{eq:pasmpoly2}). Now $X = \sum_\gamma \mu_\gamma M_\gamma$ where $\mu_\gamma \geq 0$, $\sum_\gamma \mu_\gamma = 1$, and the $M_\gamma \in \pasm_{m,n}$. Since we have a convex combination of partial alternating sign matrices, by Definition~\ref{pasm}, we obtain (\ref{eq:pasmpoly1}) and (\ref{eq:pasmpoly2}) immediately. Thus $\pasm(m,n)$ fits the inequality description.

Let $X$ be a real-valued $m \times n$ matrix satisfying (\ref{eq:pasmpoly1}) and (\ref{eq:pasmpoly2}). We wish to show that $X$ can be written as a convex combination of partial alternating sign matrices in $\pasm_{m,n}$, so that $X$ is in $\pasm(m,n)$.

Consider the corresponding labeled graph $\hat{X}$ of Definition~\ref{def:Xhat}. By assumption, all labels $\alpha$ of $\hat{X}$ satisfy $0\leq \alpha\leq 1$, since the labels equal partial row and column sums. Furthermore, by Definition~\ref{pasm}, the labels in $\hat{X}$ are all $0$ or $1$ if and only if $X$ is already a partial alternating sign matrix. So if $X$ is not a partial alternating matrix, there is at least one label strictly between $0$ and $1$. We will construct a trail in $\hat{X}$ all of whose edges are labeled by numbers that are strictly between $0$ and $1$ and show it is a simple path or cycle. We then use this path to express $X$ as a convex combination of two matrices that are both `closer' to being partial alternating sign matrices, in the sense that they will each have at least one more partial sum equal to $0$ or $1$.

Recall the notation for row and column partial sums, $r_{ij}$ and $c_{ij}$, from Definition~\ref{def:Xhat}. In addition, set $r_{i0}=0=c_{0j}$ for all $i,j$. Then for all $1\leq i\leq \lambda_1, 1\leq j\leq n$, we have $X_{ij}=r_{ij}-r_{i, j-1}=c_{ij}-c_{i-1,j}$.  Thus,
\begin{equation}
\label{eq:rc}
r_{ij}+c_{i-1,j}=c_{ij}+r_{i, j-1}.
\end{equation}

If there exists $i$ or $j$ such that boundary edge label $r_{i,n}$ or $c_{m,j}$ is strictly between $0$ and $1$, begin constructing the trail at the adjacent boundary vertex. If no such $i$ or $j$ exist, start the trail on any vertex, say $(i,j)$ adjacent to an edge with label strictly between $0$ and $1$. By (\ref{eq:rc}), at least one other adjacent edge label
is also strictly between $0$ and $1$, so we may begin forming a trail by moving through edges with  labels strictly between $0$ and $1$. From the starting point, construct the trail as follows. Go along a row or column from the starting vertex along edges with labels strictly between $0$ and $1$.
Continue in this manner until either (1) you reach a vertex adjacent to an edge that was previously in the trail, or (2) you reach a new boundary vertex. 
If (1), then the part of the trail constructed between the first and second time you reached that vertex will be a simple cycle. That is, we cut off any part that was constructed before the first time that vertex was reached. 
If (2), then the starting point for the trail must have been a boundary vertex, since there exists at least one boundary edge with label strictly between $0$ and $1$. Thus our trail is actually a path.

Label the corner vertices of the path or cycle (not the boundary vertices) alternately $(+)$ and $(-)$. Set $\ell^+$ equal to the largest positive number that we could subtract from the entries of $X$ corresponding to the $(-)$ vertices and add to the entries of $X$ corresponding to the $(+)$ vertices while still satisfying (\ref{eq:pasmpoly1}) and (\ref{eq:pasmpoly2}). Such an $\ell^+$ exists since the path or cycle was constructed such that the edge labels in $\hat{X}$ are strictly between $0$ and $1$. This means, the corresponding partial sums of $X$ are also strictly between $0$ and $1$. Construct a matrix $X^+$ by subtracting and adding $\ell^+$ to the specified entries of $X$ in this way and leaving all other entries fixed. $X^+$ is a matrix which stills satisfies (\ref{eq:pasmpoly1}) and (\ref{eq:pasmpoly2}) and which has least one more partial row or column sum equal to $0$ or $1$ than $X$ does.

Now give opposite labels to the corner vertices of the path or cycle and set $\ell^-$ equal to the largest positive number we could subtract from the entries of $X$ corresponding to the $(-)$ vertices and add to the entries of $X$ corresponding to the $(+)$ vertices while still satisfying (\ref{eq:pasmpoly1}) and (\ref{eq:pasmpoly2}). Add and subtract in a similar way to create $X^-$, another matrix satisfying (\ref{eq:pasmpoly1}) and (\ref{eq:pasmpoly2}) and which has least one more partial row or column sum equal to $0$ or $1$ than $X$ does.

Both $X^+$ and $X^-$ satisfy (\ref{eq:pasmpoly1}) and (\ref{eq:pasmpoly2}) by construction. Also by construction, \[X=\frac{\ell^-}{\ell^++\ell^-}X^++\frac{\ell^+}{\ell^++\ell^-}X^-,\] $\frac{\ell^-}{\ell^++\ell^-}$ and $\frac{\ell^+}{\ell^++\ell^-}$ are positive, and $\frac{\ell^-}{\ell^++\ell^-} + \frac{\ell^+}{\ell^++\ell^-} = 1$. So $X$ is a convex combination of the two matrices $X^+$ and $X^-$ that still satisfy the inequalities and are each at least one step closer to being partial alternating sign matrices, since they have at least one more partial sum attaining its maximum or minimum bound. By repeatedly applying this procedure, $X$ can be written as a convex combination of partial alternating sign matrices.

See Figure~\ref{pasmcircuitfig} and Example~\ref{pasmcircuitex} for an example of this construction.
\end{proof}

\begin{figure}[hbtp]
\includegraphics[scale=0.4, valign=m]{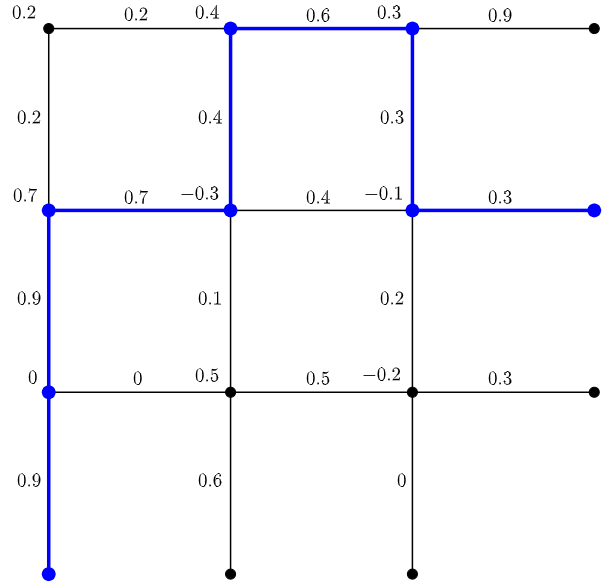}

\vspace{0.5cm}

$\swarrow$ \hspace{1.75in} $\searrow$

\vspace{0.5cm}

\includegraphics[scale=0.4, valign=m]{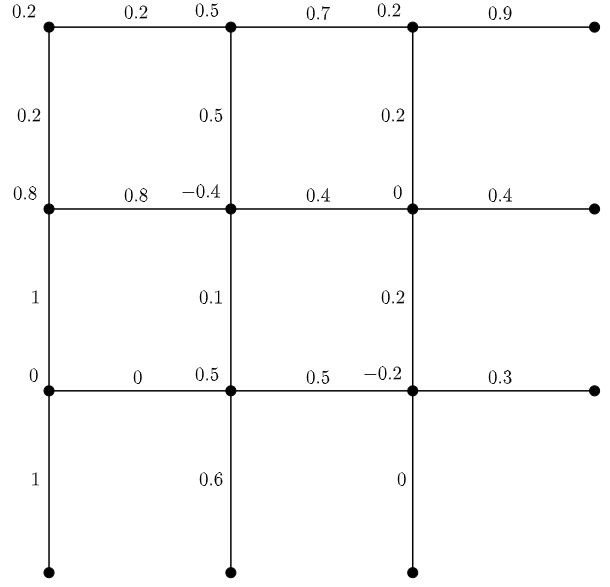} \hspace{0.5in}
\includegraphics[scale=0.4, valign=m]{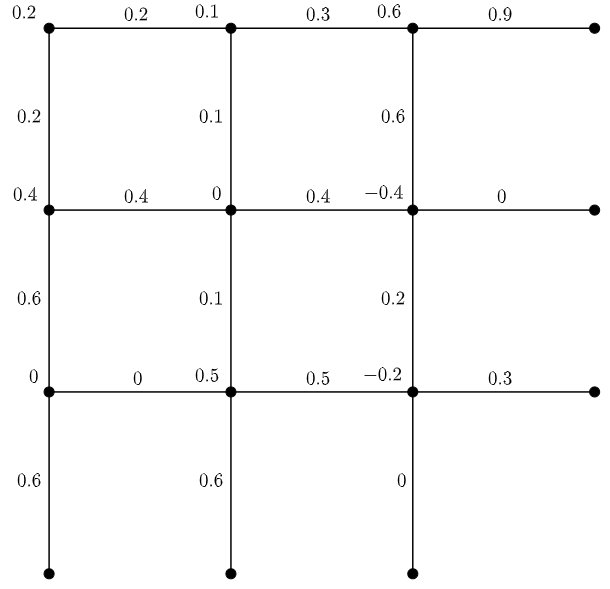}
\caption{An example of the path construction described in the proof of Theorem~\ref{thm:pasmverts}. The blue edges and vertices are those included in the path. Here we have labeled the interior vertices with their corresponding matrix entry.}
\label{pasmcircuitfig}
\end{figure}

\begin{example}
\label{pasmcircuitex}
Let $X = \begin{pmatrix} 0.2 & 0.4 & 0.3 \\ 0.7 & -0.3 & -0.1 \\ 0 & 0.5 & -0.2 \end{pmatrix}$. Then by the construction described in the proof of Theorem~\ref{thm:pasmverts} and shown in Figure~\ref{pasmcircuitfig}, $X$ can be decomposed as $X = \frac{0.3}{0.1+0.3}X^+ + \frac{0.1}{0.1+0.3}X^-$, where $X^+ = \begin{pmatrix} 0.2 & 0.5 & 0.2 \\ 0.8 & -0.4 & 0 \\ 0 & 0.5 & -0.2 \end{pmatrix}$ and $X^- = \begin{pmatrix} 0.2 & 0.1 & 0.6 \\ 0.4 & 0 & -0.4 \\ 0 & 0.5 & -0.2 \end{pmatrix}$. In this step of decomposing, $\ell^+ = 0.1$ and $\ell^- = 0.3$. Continuing the process of decomposition, one could write $X$ as a convex combination of partial alternating sign matrices.
\end{example}

Theorem~\ref{thm:pasmverts} gives a simple inequality description, but it is not a \emph{minimal} inequality description. That is, some of the inequalities in (\ref{eq:pasmpoly1}) and (\ref{eq:pasmpoly2}) are redundant. In the following theorem, we determine these redundancies to count the inequalities that determine facets.

\begin{theorem}\label{pasmfacets}
The number of facets of $\pasm(m,n)$ equals $4mn - 3m - 3n + 5$.
\end{theorem}

\begin{proof}
We claim a minimal inequality description is the following. 
\begin{align}
0  &\leq \displaystyle\sum_{i'=1}^{i} X_{i'j}, & \mbox{ for all }  1 \leq i \leq m, \mbox{ } 2 \leq j \leq n, \mbox{ and } i=j=1  \label{ineq1a} \\
\displaystyle\sum_{i'=1}^{i} X_{i'j} & \leq 1, & \mbox{ for all }  2 \leq i \leq m, \mbox{ } 2 \leq j \leq n, \mbox{ and }   i=m,j=1 \label{ineq1b} \\
0 &\leq \displaystyle\sum_{j'=1}^{j} X_{ij'}, & \mbox{ for all }  2 \leq i \leq m, \mbox{ } 1 \leq j \leq n,  \label{ineq2a} \\
 \displaystyle\sum_{j'=1}^{j} X_{ij'} & \leq 1, & \mbox{ for all }  2 \leq i \leq m, \mbox{ } 2 \leq j \leq n, \mbox{ and } i=1,j=n \label{ineq2b} 
\end{align}

We prove this by first showing the additional inequalities from Theorem~\ref{thm:pasmverts} are implied by those listed above. Then we will show inequalities (\ref{ineq1a})--(\ref{ineq2b}) are irredundant.

Note that (\ref{ineq1a})--(\ref{ineq2b}) would be exactly the same $4mn$ inequalities as (\ref{eq:pasmpoly1}) and (\ref{eq:pasmpoly2}) if all the ranges for $i$ and $j$ were $1\leq i\leq m$ and $1\leq j\leq n$. We show how each omitted combination of $i$ and $j$ is implied by other inequalities in (\ref{ineq1a})--(\ref{ineq2b}).

First, we note the inequality of (\ref{ineq1a}) in the case $i=1$ is $0 \leq X_{1j}$. This implies that $0 \leq \sum_{j'=1}^{j} X_{1j'}$ for $1 \leq j \leq n$, which shows why $i=1$ is not included in (\ref{ineq2a}).

When $i=1, j=n$, the inequalities of (\ref{ineq1a}) and (\ref{ineq2b}) give $0 \leq X_{1n}$ and $\sum_{j'=1}^n X_{1j'} \leq 1$, respectively. This implies that $\sum_{j'=1}^{n-1} X_{1j'} \leq 1 - X_{1n} \leq 1$. Similarly, the inequalities of $(\ref{ineq1a})$ when $i=1$ imply the $n-1$ inequalities of the form $\sum_{j'=1}^j X_{1j'} \leq 1$ for $1 \leq j < n$. This shows why $i=1$ is not included in (\ref{ineq2b}) except when $j=n$.

We now use the redundant inequalities shown in the previous two paragraphs: $\sum_{j'=1}^{j-1} X_{1j'} \geq 0$ and $\sum_{j'=1}^j X_{1j'} \leq 1$ for $2 \leq j \leq n$. Together these imply that $X_{1j} \leq 1 - \sum_{j'=1}^{j-1} X_{1j'} \leq 1$ for $2 \leq j \leq n$, so we omit $i=1$ in (\ref{ineq1b}).

The inequality of (\ref{ineq1a}) when $i=j=1$, and the inequality of (\ref{ineq2a}) when $i=2, j=1$, together imply that $0 \leq X_{11}+X_{21}$. Similarly, the inequality of (\ref{ineq2a}) in the case $j=1$ implies that $0 \leq \sum_{i'=1}^{i} X_{i'1}$ for $2 \leq i\leq m$. This is why $j=1$ is omitted in (\ref{ineq1a}) except in the case when $i=1$.

When $i=m, j=1$, the inequalities of (\ref{ineq1b}) and (\ref{ineq2a}) give $\sum_{i'=1}^m X_{i'1} \leq 1$ and $0 \leq X_{m1}$, respectively. This implies that $\sum_{i'=1}^{m-1} X_{i'1} \leq 1 - X_{m1} \leq 1$. Similarly, the inequalities of (\ref{ineq2a}) when $j=1$ imply the $m-1$ inequalities of the form $\sum_{i'=1}^{i} X_{i'1} \leq 1$ for $1 \leq i < m$. This shows why $j=1$ is not included in (\ref{ineq1b}) except when $i=m$. 

We now use the redundant inequalities shown in the previous two paragraphs: $\sum_{i'=1}^{i-1} X_{i'1} \geq 0$ and $\sum_{i'=1}^i X_{i'1} \leq 1$ for $2 \leq i \leq m$. Together these imply that $X_{i1} \leq 1 - \sum_{i'=1}^{i-1} X_{i'1} \leq 1$ for $2 \leq i \leq m$, so we omit $j=1$ in (\ref{ineq2b}).

Overall, this means that the number of facets is at most $4mn - 3m - 3n + 5$, each made by changing one of the inequalities in (\ref{ineq1a})--(\ref{ineq2b}) to an equality. We claim that this upper bound is the facet count. That is, a facet can be defined 
as the set of all $X \in \pasm(m,n)$ which satisfy exactly one of the following:

\begin{align}
0  & = \displaystyle\sum_{i'=1}^{i} X_{i'j}, & \mbox{ for all }  1 \leq i \leq m, \mbox{ } 2 \leq j \leq n, \mbox{ and } i=j=1  \label{eq1a} \\
\displaystyle\sum_{i'=1}^{i} X_{i'j} & = 1, & \mbox{ for all }  2 \leq i \leq m, \mbox{ } 2 \leq j \leq n, \mbox{ and }   i=m,j=1 \label{eq1b} \\
0 & = \displaystyle\sum_{j'=1}^{j} X_{ij'}, & \mbox{ for all }  2 \leq i \leq m, \mbox{ } 1 \leq j \leq n,  \label{eq2a} \\
 \displaystyle\sum_{j'=1}^{j} X_{ij'} & = 1, & \mbox{ for all }  2 \leq i \leq m, \mbox{ } 2 \leq j \leq n, \mbox{ and } i=1,j=n \label{eq2b} 
\end{align}

To show this, let two generic equalities of the form (\ref{eq1a})--(\ref{eq2b}) be denoted as $\alpha_{ij}$ and $\beta_{k\ell}$, where the indices $(i,j)$ and $(k,\ell)$ must be in the corresponding ranges indicated by (\ref{eq1a})--(\ref{eq2b}). In the cases below, we will construct an $m \times n$ partial alternating sign matrix $M$, such that $M$ satisfies $\alpha_{ij}$ and not $\beta_{k\ell}$.

\smallskip
\textbf{\emph{Case 1:}} $\alpha_{ij}$ is an equality in (\ref{eq1a}) or (\ref{eq2a}) and $\beta_{k\ell}$ is an equality in (\ref{eq1b}) or (\ref{eq2b}).
We set $M$ equal to the zero matrix.

\smallskip
In each of the following, we will specify the nonzero entries of $M$, and assume all other entries are zero.

\textbf{\emph{Case 2:}} $\alpha_{ij}$ and $\beta_{k\ell}$ are in (\ref{eq1a}) or (\ref{eq2a}).
\begin{itemize}
\item If $i \neq k$ and $j \neq \ell$ let $M_{k\ell}=1$.
\item
Suppose $\alpha_{ij}$ and $\beta_{k\ell}$ are both in (\ref{eq1a}). If $j \neq \ell$, let $M_{k\ell}=1$. If $j=\ell$ and $i<k$, let $M_{k\ell}=1$. If $j=\ell$ and $i>k$, let $M_{k\ell}=M_{k+1,\ell-1}=1$ and $M_{k+1,\ell}=-1$.
\item
Suppose $\alpha_{ij}$ and $\beta_{k\ell}$ are both in (\ref{eq2a}). If $i \neq k$, let $M_{k\ell}=1$. If $i=k$ and $j<\ell$, let $M_{k\ell}=1$. If $i=k$ and $j>\ell$, let $M_{k\ell}=M_{k-1,\ell+1}=1$ and $M_{k+1,\ell}=-1$.
\item
If $\alpha_{ij}$ is in (\ref{eq2a}) and $\beta_{k\ell}$ is in (\ref{eq1a}), let $M_{1\ell}=1$.
\item
If $\alpha_{ij}$ is in (\ref{eq1a}) and $\beta_{k\ell}$ is in (\ref{eq2a}), let $M_{k1}=1$.
\end{itemize}

\textbf{\emph{Case 3:}}  $\alpha_{ij}$ and $\beta_{k\ell}$ are in (\ref{eq1b}) or (\ref{eq2b}).

\begin{itemize}
\item 
If $i \neq k$ and $j \neq \ell$, let $M_{ij}=1$.
\item
Suppose $\alpha_{ij}$ and $\beta_{k\ell}$ are both in (\ref{eq1b}). If $j \neq \ell$, let $M_{ij}=1$. If $j=\ell$ and $i < k$, let $M_{ij}=M_{i+1,j-1}=1$ and $M_{i+1,j}=-1$ If $j=k$ and $i > \ell$, let $M_{ij} = 0$.
\item
Suppose $\alpha_{ij}$ and $\beta_{k\ell}$ are both in (\ref{eq2b}). If $i \neq k$, let $M_{ij}=1$. If $i=k$ and $j<\ell$, let $M_{ij}=M_{i-1,j+1}=1$ and $M_{i,j+1}=-1$. If $i=k$ and $j>\ell$, let $M_{ij}=1$.
\item
If $\alpha_{ij}$ is in (\ref{eq2b}) and $\beta_{k\ell}$ is in (\ref{eq1b}), let $M_{1j}=1$.
\item
If $\alpha_{ij}$ is in (\ref{eq1b}) and $\beta_{k\ell}$ is in (\ref{eq2b}), let $M_{i1}=1$.
\end{itemize}

\textbf{\emph{Case 4:}}  $\alpha_{ij}$ is in (\ref{eq1b}) or (\ref{eq2b}) and $\beta_{k\ell}$ is in (\ref{eq1a}) or (\ref{eq2a}).

\begin{itemize}
\item If $i \neq k$ and $j \neq \ell$, let $M_{ij}=M_{k\ell}=1$.
\item
Suppose $\alpha_{ij}$ is in (\ref{eq1b}) and $\beta_{k\ell}$ is in (\ref{eq1a}). If $i=k$ and $j \neq \ell$, let $M_{ij}=M_{1\ell}=1$. If $j=\ell$, let $M_{1j}=1$.
\item
Suppose $\alpha_{ij}$ is in (\ref{eq2b}) and $\beta_{k\ell}$ is in (\ref{eq2a}). If $i=k$, let $M_{i1}=1$. If $j=\ell$ and $i \neq k$, let $M_{ij}=M_{k1}=1$.
\item
Suppose $\alpha_{ij}$ is in (\ref{eq2b}) and $\beta_{k\ell}$ is in (\ref{eq1a}). If $i=k$ and $j < \ell$, let $M_{ij}=M_{1\ell}=1$. If $i=k$ and $j>\ell$, let $M_{k\ell}=1$. If $j=\ell$ and $i \leq k$, let $M_{ij}=1$. If $j=\ell$ and $i < k$, let $M_{k\ell}=M_{i1}=1$ .
\item
Suppose $\alpha_{ij}$ is in (\ref{eq1b}) and $\beta_{k\ell}$ is in (\ref{eq2a}).  If $i=k$ and $j\leq \ell$, let $M_{ij}=1$. If $i=k$ and $j>\ell$, let $M_{k\ell}=M_{1j}=0$. If $j=\ell$ and $i<k$, let $M_{ij}=M_{k1}=1$. If $j=\ell$ and $i>k$, let $M_{k\ell}=1$.
\end{itemize}

In each of these cases, $M$ is constructed so that it satisfies $\alpha_{ij}$ but not $\beta_{ij}$, so each of the equalities in (\ref{eq1a})--(\ref{eq2b}) gives rise to a unique facet. Thus there are $4mn - 3m -3n +5$ facets of $\pasm(m,n)$.
\end{proof}

\begin{remark}
\label{unimod}
The above inequality description may make one wonder whether the matrix defining $\pasm(m,n)$ is \emph{totally unimodular}. Consider the case when $m=n=2$. Then there are $3 \times 3$ submatrices with determinant $2$ and $-2$, so the matrix is not totally unimodular.
\end{remark}

\subsection{Face lattice}
\label{subsec:pasmfacelattice}
In this subsection, we characterize the face lattice of $\pasm(m,n)$ in Theorem~\ref{thm:pasm_facelattice}, using \emph{sum-labelings} of the graph $\Gamma(m,n)$ (see Definition~\ref{def:gamma}).

Recall from Remark~\ref{remark:signmatrices} that partial alternating sign matrices are a subset of sign matrices~\cite{SolhjemStriker}. It was shown in~\cite[Theorem 5.3]{SolhjemStriker} that the convex hull of $m\times n$ sign matrices, denoted $P(m,n)$, has inequality description as in Theorem~\ref{thm:pasmverts}, except in (\ref{eq:pasmpoly2}) the $\leq  1$ is not present. More specifically, we have the following relation.

\begin{lemma}
\label{lem:PmnPASMmn}
The polytope $\pasm(m,n)$ is the intersection: 
\[P(m,n)\cap\displaystyle\bigcap_{1 \leq i \leq m, 1 \leq j \leq n}H_{ij},\]
where $H_{ij}$ is the closed halfspace of $m\times n$ real matrices $X=(X_{ij})$ such that $\displaystyle\sum_{i'=1}^{i} X_{i'j}  \leq 1$.
\end{lemma}

We now state some definitions and a lemma that will help prove Theorem~\ref{thm:pasm_facelattice} describing the face lattice of $\pasm(m,n)$. This theorem is analogous to \cite[Theorems 7.15 and 7.16]{SolhjemStriker} which describe the face lattice of $P(m,n)$. The proof is also similar.

Recall $\hat{M}$ from Definition~\ref{def:Xhat}.
\begin{definition}
\label{sum-labeling}
A \emph{basic sum-labeling} of ${\Gamma}_{(m,n)}$ is a labeling of the edges of $\Gamma_{(m,n)}$ with $0$ or $1$ such that the edge labels equal the corresponding edge labels of $\hat{M}$ for some $M\in\pasm_{m,n}$. 
\end{definition}

\begin{remark}
Recall we can recover any matrix from its column partial sums. Thus basic sum-labelings of ${\Gamma}_{(m,n)}$ are in bijection with partial alternating sign matrices $\pasm_{m,n}$. This is a linear isomorphism, so $\pasm(m,n)$ is linearly isomorphic to a $0/1$ polytope.
\end{remark}

\begin{definition}
\label{def:union}
Let $\delta$ and $\delta'$ be labelings of the edges of  $\Gamma_{(m,n)}$ with $\{0\}$, $\{1\}$, or $\{0,1\}$.
Define the \emph{union} $\delta \cup \delta'$  as the labeling of $\Gamma_{(m,n)}$ such that each edge is labeled by the union of the corresponding labels on $\delta$ and $\delta'$. Define  \emph{intersection} $\delta\cap\delta'$ and \emph{containment} $\delta\subseteq \delta'$  similarly. 
\end{definition}

\begin{definition}
A \emph{sum-labeling} $\delta$ of ${\Gamma}_{(m,n)}$ is either the empty labeling of $\Gamma_{(m,n)}$ 
(denoted $\emptyset$) 
or a labeling of the edges of $\Gamma_{(m,n)}$ with $\{0\}$, $\{1\}$, or $\left\{0,1\right\}$ such
that there exists a set $S$ of basic sum-labelings of $\Gamma_{(m,n)}$ so that $\delta = \bigcup_{\delta'\in S} \delta'$.
\end{definition}

\begin{definition} 
Given $M\in\pasm_{m,n}$,
let $g(M)$ denote the basic sum-labeling of $\Gamma_{(m,n)}$ associated to $M$. 
Given a collection of partial alternating sign matrices $\mathcal{M}=\{M_1,M_2,\dots,M_r\}\subseteq\pasm_{m,n}$, 
define the map $g(\mathcal{M})=\displaystyle\bigcup_{i=1}^r g(M_i)$.
\end{definition}

\begin{definition}
\label{def:region}
Given a sum-labeling $\delta$, consider the planar graph $G$ composed of the edges of $\delta$ labeled by the two-element set $\{0,1\}$ (and all incident vertices), where we regard any external edges on the right and bottom as meeting at a point in the exterior.
A \emph{region} of $\delta$ is defined as a planar region of $G$, excluding the exterior region. 
Let $\mathcal{R}(\delta)$ denote the number of regions of $\delta$. (For consistency we set $\mathcal{R}(\emptyset)=-1$.)
\end{definition}

See Figure~\ref{fig:sum-labeling} for an example of a sum-labeling of $\Gamma_{(2,3)}$ with $4$ regions.

\begin{figure}[hbtp]
\centering
\includegraphics[scale=0.5]{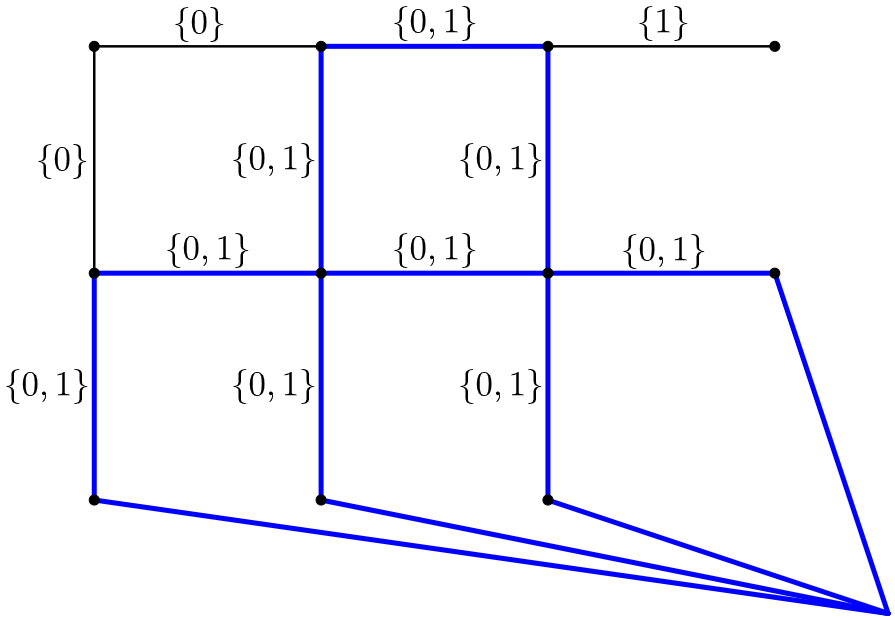}
\caption{The sum-labeling of $\Gamma_{(2,3)}$ which is $g(M_3) \cup g(M_{13}) \cup g(M_{15})$, where $M_3$, $M_{13}$, and $M_{15}$ are as in Examples~\ref{ex:pperms} and \ref{ex:pasms}. Edges labeled $\{0,1\}$ are colored blue to accentuate the regions.}
\label{fig:sum-labeling}
\end{figure}

\begin{lemma}
\label{prop:regions}
Consider sum-labelings $\delta$ and $\delta'$. If $\delta \subset \delta'$ (where $\subset$ denotes strict containment), then $\mathcal{R}(\delta)<\mathcal{R}(\delta')$. 
\end{lemma}

\begin{proof}
By convention, the empty labeling has $\mathcal{R}(\emptyset)=-1$. 
If $\delta$ is a basic sum-labeling, $\mathcal{R}(\delta)=0$, as there are no edges labeled $\{0,1\}$ in a basic sum-labeling. Suppose a sum-labeling $\delta$ has $\mathcal{R}(\delta)=\omega>0$. We wish to show if $\delta\subset\delta'$ then $\mathcal{R}(\delta')> \omega$. Now $\delta\subset\delta'$ implies that the labels of each edge of $\delta$ are subsets of the labels of each edge of $\delta'$, where at least one of these containments is strict. So there is an edge in $\delta'$ labeled $\{0,1\}$ that was labeled $\{0\}$ or $\{1\}$ in $\delta$. So $\delta'$ contains a basic sum labeling $\beta'$ that differs from all the basic sum labelings in $\delta$ at edge $e$. 
Let $\beta$ denote a basic sum labeling such that $\beta\subseteq\delta$.
By Equation~(\ref{eq:rc}), at least one edge label of $\beta'$ adjacent to $e$  must also differ from the corresponding edge label of $\beta$. By iterating this (as in the proof of Theorem~\ref{thm:pasmverts}), $\beta'$ differs from $\beta$ by at least one simple path (connecting boundary vertices) or cycle of differing partial sums. This path or cycle appears as edges labeled by $\{0,1\}$ in $\delta'$, and at least one of these edges was not labeled by $\{0,1\}$ in $\delta$. So $\delta'$ has at least one new region. Therefore, $\mathcal{R}(\delta')> \omega$. 
\end{proof}

We are now ready to state and prove the main theorem of this subsection.
\begin{theorem}
\label{thm:pasm_facelattice}
Let $F$ be a face of $\pasm(m,n)$ and $\mathcal{M}(F)$ be the set of partial alternating sign matrices that are vertices of $F$. The map $\psi:F\mapsto g(\mathcal{M}(F))$ induces an isomorphism between the face lattice of $\pasm(m,n)$ and the set of sum-labelings of ${\Gamma}_{(m,n)}$ ordered by containment. Moreover, 
$\mathrm{dim}(F)= \mathcal{R}(\psi(F))$.
\end{theorem}

\begin{proof}
Let $F$ be a face of $\pasm(m,n)$. Then $g(\mathcal{M}(F))$ is a sum-labeling of ${\Gamma}_{(m,n)}$ since $g(\mathcal{M}(F))=\displaystyle\bigcup_{i=1}^r g(M_i)$ is a union of basic sum-labelings.
We now construct the inverse of $\psi$, call it $\varphi$. Given a sum-labeling $\nu$ of ${\Gamma}_{(m,n)}$, let $\varphi(\nu)$ be the face that results as the intersection of the facets corresponding to the edges of $\nu$ with label $0$ or $1$. 

We wish to show $\psi(\varphi(\nu))=\nu$. First, we show $\nu\subseteq \psi(\varphi(\nu))$. Let $M \in \pasm(m,n)$ such that $g(M) \subset \nu$ is a basic sum-labeling. Then $M$ is in the intersection of the facets that yields $\varphi(\nu)$, since otherwise $g(M)$ would not be a basic sum-labeling such that $g(M) \subset \nu$. Thus $g(M) \subseteq \psi(\varphi(\nu))$ as well. So $\nu\subseteq \psi(\varphi(\nu))$.

Next, we show $\psi(\varphi(\nu)) \subseteq \nu$. Suppose not. Then there exists some edge $e$ of $\Gamma_{(m,n)}$ whose label in $\psi(\varphi(\nu))$  strictly contains the  label of $e$ in $\nu$. The label of $e$ in $\nu$ is $0$ or $1$ and the label of $e$ in $\psi(\varphi(\nu))$ is $\{0,1\}$. Let $\gamma$ denote the label of $e$ in $\nu$. As in the previous case, the facet corresponding to the label $\gamma$ on $e$ would have been one of the facets intersected to get $\varphi(\nu)$. Therefore the matrix partial column sum corresponding to edge $e$ would be fixed as $\gamma$ in each partial alternating sign matrix in $\varphi(\nu)$. So in the union $\psi(\varphi(\nu))$, that edge label would be the union of the edge labels of all the partial alternating sign matrices in $\varphi(\nu)$, and this union would be $\gamma$. This is a contradiction. Thus $\nu=\psi(\varphi(\nu))$.

Let $F_1$ and $F_2$ be faces of $\pasm(m,n)$ such that $F_1 \subset F_2$. Then $F_1$ is an intersection of $F_2$ and some facet hyperplanes. In other words, $F_1$ is obtained from $F_2$ by setting at least one of the inequalities in Theorem~\ref{thm:pasmverts} to an equality. We have that $\psi(F_1)$ is obtained from $\psi(F_2)$ by changing at least edge label of $\{0,1\}$ to a label of $0$ or $1$. Therefore we have $\psi(F_1) \subset \psi(F_2)$.

Conversely, suppose that $\psi(F_1) \subset \psi(F_2)$. Recall the inverse of $\psi$ is $\varphi$, where for any sum-labeling $\nu$ of $\Gamma_{(m,n)}$, $\varphi(\nu)$ is the face of $\pasm(m,n)$ that results as the intersection of the facets corresponding to the edges of $\nu$ with labels $0$ or $1$. Now if $\psi(F_1) \subset \psi(F_2)$, the edges of $\psi(F_1)$ with label $\{0,1\}$ are a subset of such edges of $\psi(F_2)$, so the edges of $\psi(F_2)$ with labels of either $0$ or $1$ are a subset of such edges of $\psi(F_1)$. So $\varphi(\psi(F_1))$ is an intersection of the facets intersected in $\varphi(\psi(F_2))$ and one or more additional facets. Thus $F_1=\varphi(\psi(F_1))\subset \varphi(\psi(F_2))=F_2$.

Now, we prove the dimension claim. Recall from Remark~\ref{pasm_dim} that $\mathrm{dim}(\pasm(m,n))=mn$. Since $\psi$ is a poset isomorphism, $\psi$ maps a maximal chain of faces $F_0 \subset F_1 \subset \cdots \subset F_{mn}$  to the maximal chain $\psi(F_0) \subset \psi(F_1) \subset \cdots \subset \psi(F_{mn})$ in the sum-labelings of ${\Gamma}_{(m,n)}$. The sum-labeling whose labels are all equal to $\{0,1\}$ contains all other sum-labelings, and this sum-labeling has $mn$ regions. Thus the result follows by Lemma~\ref{prop:regions}.
\end{proof}

\begin{remark}
The Birkhoff polytope is not only nice combinatorially, but its face lattice description in terms of matchings represents a fundamental problem in combinatorial optimization.
Though we do not discuss it here, we note that the study of the combinatorial optimization problem corresponding to linear programming on $\pasm(m,n)$ would be worthwhile.
\end{remark}

\subsection{Volume}
\label{subsec:pasmvolume}
The normalized volume of  $\pasm(m,n)$ for small values of $m$ and $n$ is given in Figure~\ref{pasm_vol} (computed in SageMath). Due to the large size of the polytopes, further computations are not easily obtained. Note that there does not appear to be a nice formula for the volume.

\begin{figure}[hbtp]
\begin{tabular}{|c|c|c|c|c|}
\hline
\diaghead(1,-1)%
   {\theadfont nnn}%
   {$m$}{$n$} & 1 & 2   & 3      & 4          \\ \hline
    1     & 1 & 1   & 1      & 1          \\ \hline
    2     & 1 & 6   & 43     & 308        \\ \hline
    3     & 1 & 43  & 5036   & 696658     \\ \hline
    4     & 1 & 308 & 696658 & 3106156252 \\ \hline
\end{tabular}
\caption{The normalized volume of $\pasm(m,n)$ for small values of $m$ and $n$.}
\label{pasm_vol}
\end{figure}

\begin{remark}
We have used SageMath to compute the Ehrhart polynomials for $\pasm(m,n)$ for $m,n \leq 4$ and note that in all of these cases their coefficients are positive.
\end{remark}

\section{Partial permutohedron}
\label{sec:ppermuto}
In this section, we study partial permutohedra that arise naturally as projections of $\pperm(m,n)$ and $\pasm(m,n)$.
After giving the definition, we count vertices and facets and find an inequality description in Subsection~\ref{ppermuto1}. Then in Subsection~\ref{ppermutofacelattice}, we note the relation between the partial permutohedron and the stellohedron and give a new combinatorial description of its face lattice. We show in Subsection~\ref{ppermutoprojection} that partial permutation and partial alternating sign matrix polytopes project to partial permutohedra. Finally, in Subsection~\ref{ppermutovolume}, we give a result and conjecture on volume.

\subsection{Vertices, facets, inequality description}
\label{ppermuto1}
In this subsection, we first give the definition of partial permutohedra. We enumerate the vertices in Proposition~\ref{ppermutoverts} and the facets in Theorem~\ref{permutofacets} and prove an inequality description in Theorem~\ref{thm:permuto_ineq}.
\begin{definition}
Given a partial permutation matrix $M\in P_{m,n}$, its \emph{one-line notation} $w(M)$ is a word $w_1 w_2 \ldots w_m$ where $w_i=j$ if there exists $j$ such that $M_{ij}=1$ and $0$ otherwise.
\end{definition}

\begin{example}
Let $M = \begin{pmatrix} 0 & 0 & 1 & 0 & 0 \\ 0 & 0 & 0 & 0 & 1 \\ 0 & 0 & 0 & 0 & 0 \\ 0 & 1 & 0 & 0 & 0\end{pmatrix}$. Then $w(M) = 3502$.
\end{example}

\begin{proposition}
The set $w(P_{m,n})$ of words of all matrices in $P_{m,n}$ can be characterized as the set of all words of length $m$ whose entries are in $\left\{0,1,\ldots,n\right\}$ and whose nonzero entries are distinct.
\end{proposition}

\begin{proof}
By definition, any matrix in $P_{m,n}$ has $m$ rows and $n$ columns with at most one $1$ in any given row or column. Thus its image under $w$ will be a word of length $m$ with entries in $\{0,1,\ldots,n\}$ such that the nonzero entries are all distinct. It follows from the definition of $w$ that this map is bijective.
\end{proof}

\begin{definition}
\label{def:partialpermutohedron}
Let \emph{$\mathcal{P}(m,n)$} be the polytope defined as the convex hull, as vectors in $\mathbb{R}^{m}$, of the words in $w(P_{m,n})$. 
Call this the \emph{$(m,n)$-partial permutohedron}.
\end{definition}

\begin{remark}
\label{ppermuto_dim}
The dimension of \emph{$\mathcal{P}(m,n)$} is $m$. To see this, let $U_{i}$ be the $m \times n$ matrix with $(i,1)$ entry equal to 1 and zeros elsewhere. Then $w(U_i)$ is the unit vector with $1$ in position $i$ and all other entries equal to $0$. Note that $U_{i} \in \mathcal{P}(m,n)$ for all $1 \leq i \leq m$. Since \emph{$\mathcal{P}(m,n)$} contains each of these $m$ unit vectors, its dimension equals the ambient dimension $m$. 
\end{remark}

\begin{definition}
\label{Pz}
Let $z \in \mathbb{R}^n$ be a vector with distinct nonzero entries.
Define $\phi_z:\mathbb{R}^{m \times n}\rightarrow\mathbb{R}^m$ as $\phi_z(X)=Xz$. Also define $w_z(P_{m,n})$ as the set of all words of length $m$ whose entries are in $\left\{0,z_1,z_2,\ldots,z_n\right\}$ and whose nonzero entries are distinct. Then $\mathcal{P}_z(m,n)$ is the polytope defined as the convex hull, as vectors in $\mathbb{R}^m$, of the words in $w_z(P_{m,n})$.
\end{definition}

Note that we will not use Definition~\ref{Pz} until Section~\ref{ppermutoprojection}, but the upcoming results about the structure of partial permutohedra can also be extended to $\mathcal{P}_z$ polytopes.

\begin{proposition}
\label{ppermutoverts}
The number of vertices of $\mathcal{P}(m,n)$ equals \begin{equation}\displaystyle\sum_{k=\max(m-n,0)}^m \frac{m!}{k!}.\end{equation}
\end{proposition}

\begin{proof}
The extreme points of $\mathcal{P}(m,n)$ are those whose nonzero entries are maximized. That is, if $k$ is the number of zeros, the $(m-k)$ nonzero entries must be precisely $\{n, n-1, \ldots, n-(m-k)+1\}$. Now, since there are $m$ total entries and $k$ zeros, there are $\frac{m!}{k!}$ distinct vectors whose $m-k$ nonzero elements are maximized.
\end{proof}

For the proof of the next theorem, and for that of Theorem~\ref{pasm-proj}, we need the concept of (weak) majorization \cite{majorization}.

\begin{definition}[\protect{\cite[Definition A.2]{majorization}}] \label{maj}
Let $u$ and $v$ be vectors of length $N$. Then $u \prec_w v$ (that is, $u$ is \emph{weakly majorized} by $v$) if
\begin{equation}
\displaystyle \sum_{i=1}^k u_{[i]} \leq \sum_{i=1}^k v_{[i]}, \mbox{ for all } 1 \leq k \leq N
\end{equation}
where the vector $\left(u_{[1]},u_{[2]},\ldots, u_{[N]}\right)$ is obtained from $u$ by rearranging its components so that they are in decreasing order (and similarly for $v$).
\end{definition}

\begin{proposition}[\protect{\cite[Proposition 4.C.2]{majorization}}] \label{maj_prop}
For vectors $u$ and $v$ of length $n$, $u \prec_w v$ if and only if $u$ lies in the convex hull of the set of all vectors $z$ which have the form $z = \left(\varepsilon_1 v_{\pi(1)},\ldots,\varepsilon_n v_{\pi(n)}\right)$, where $\pi$ is a permutation and each $\varepsilon_i$ is either $0$ or $1$.
\end{proposition}

\begin{theorem}
\label{thm:permuto_ineq}
The polytope $\mathcal{P}(m,n)$ consists of all vectors $u\in\mathbb{R}^m$  such that:
\begin{align}
\label{eq:permuto1}
\displaystyle\sum_{i \in S} u_{i} & \leq \binom{n+1}{2}-\binom{n-k+1}{2}, & \mbox{ where } S \subseteq \{1, \ldots, m\}, |S| = k \neq 0 , \mbox{ and } \\
\label{eq:permuto2}
u_i & \geq 0, & \mbox{ for all } 1 \leq i \leq m.
\end{align}
\end{theorem}

\begin{proof}
First, note that if $P \in P_{m,n}$, then $w(P)$ satisfies (\ref{eq:permuto1}) and (\ref{eq:permuto2}). This is because the largest values that may appear are the $m$ largest non-negative integers less than or equal to $n$, and the nonzero integers must be distinct. Since $w(P)$ satisfies the inequalities for any $P$, so must any convex combination.

Now, suppose $x \in \mathbb{R}^m$ satisfies (\ref{eq:permuto1}) and (\ref{eq:permuto2}). We will proceed by using Proposition~\ref{maj_prop}. Fix $n$ and let $v = (n, n-1, n-2, \ldots, 1, 0, \ldots, 0)$ be the decreasing vector in $\mathbb{R}^m$ whose largest entry is $n$, and whose subsequent nonzero entries decrease by $1$ and for which all other entries are $0$. Note that if $n \geq  m$, then $v$ will have no $0$ entries: it will be $(n, n-1, \ldots, n-m+1)$. Since $x$ satisfies (\ref{eq:permuto1}) and (\ref{eq:permuto2}), it is by definition weakly majorized by $v$; note in particular that (\ref{eq:permuto1}) requires that the sum of the $k$ largest entries is never more than the $k$ largest integers less than or equal to $n$. But now the convex hull described in Proposition~\ref{maj_prop} is actually $\mathcal{P}(m,n)$, thus $x \in \mathcal{P}(m,n)$.
\end{proof}

\begin{theorem} \label{permutofacets}
The number of facets of $\mathcal{P}(m,n)$ equals $m + 2^m -1 - \displaystyle\sum_{r=1}^{m-n}\binom{m}{m-r}$.
\end{theorem}

\begin{proof}
There are $2^m-1$ total inequalities given in (\ref{eq:permuto1}), and $m$ inequalities given in (\ref{eq:permuto2}). Note that $\binom{n-k+1}{2}=0$ whenever $k\geq n$. When $m>n$, there are $m-n+1$ values of $k$ such that $\binom{n-k+1}{2}=0$, creating redundancies.
For each $r$ between $1$ and $m-n$, we have redundant inequalities for the subsets of $\{1,\ldots,m\}$ of size $m-r$. These are counted by $\binom{m}{m-r}$. 

When $m \leq n$, none of the inequalities in (\ref{eq:permuto1}) are redundant, since $\binom{n-k+1}{2} = 0$ may only be satisfied by $k=n$.  
\end{proof}

\begin{remark}
When $m \geq n$, the number of facets of $\mathcal{P}(m,n)$ can also be written as: \[m + \displaystyle\sum_{r = m-n+1}^m \binom{m}{m-r}.\]
\end{remark}

\subsection{Face lattice}
\label{ppermutofacelattice}
In this subsection, we give a combinatorial description of the face lattice of $\mathcal{P}(m,m)$ in Theorem~\ref{chain-face-lattice} involving chains in the Boolean lattice. We furthermore state Conjecture~\ref{mn-chain-face}, which extends this characterization to $m\neq n$.

We begin by relating $\mathcal{P}(m,m)$ to a specific graph associahedron, the stellohedron. But first, we need the following definitions.
\begin{definition}[\protect{\cite[Definition 2.2]{Devadoss1}}] \label{tubes}

Let $G$ be a connected graph. A \emph{tube} is a proper nonempty set of vertices of $G$ whose induced graph is a proper, connected subgraph of $G$. There are three ways that two tubes $t_1$ and $t_2$ may interact on the graph:
\begin{enumerate}
\item Tubes are \emph{nested} if $t_1 \subset t_2$.
\item Tubes \emph{intersect} if $t_1 \cap t_2 \neq \emptyset$, $t_1 \not\subset t_2$, and $t_2 \not\subset t_1$.
\item Tubes are \emph{adjacent} if $t_1 \cap t_2 = \emptyset$ and $t_1 \cup t_2$ is a tube in $G$.
\end{enumerate}

Tubes are \emph{compatible} if they do not intersect and they are not adjacent. A \emph{tubing $T$} of $G$ is a set of tubes of $G$ such that every pair of tubes is compatible. A \emph{$k$--tubing} is a tubing with $k$ tubes.
\end{definition}

\begin{definition}[\protect{\cite[Definition 2]{Devadoss2}}] \label{graph-assoc}
For a graph $G$, the \emph{graph associahedron} $\text{Assoc}(G)$ is a simple, convex polytope whose face poset is isomorphic to the set of tubings of $G$, ordered such that $T < T'$ if $T$ obtained from $T'$ by adding tubes.
\end{definition}

Of particular interest to us is the graph associahedron of the star graph, $\text{Assoc}(K_{1,m})$, also called the \emph{stellohedron}.

\begin{definition}
The \emph{star graph} (with $m+1$ vertices) is the complete bipartite graph $K_{1,m}$. We label the lone vertex $*$, and call it the \emph{inner vertex}. We label the other $m$ vertices $x_1, x_2, \ldots, x_m$, and call them \emph{outer vertices}.
\end{definition}

\begin{remark}
Note that if $G$ has $n$ nodes, vertices of $\text{Assoc}(G)$ correspond to maximal tubings of $G$ (\emph{i.e.} $(n-1)$--tubings), and in general, faces of dimension $k$ correspond to $(n-k-1)$-tubings of $G$. Thus for the star graph $K_{1,m}$, which has $m+1$ nodes, vertices of $\text{Assoc}(K_{1,m})$ correspond to $m$-tubings, and in general, faces of dimension $k$ correspond to $(m-k)$--tubings.
\end{remark}

We examine the polytope $\text{Assoc}(K_{1,m})$ through the lens of partial permutations, which allows us to understand it in a different way. Lemmas~\ref{tubings-chains-bijection} and \ref{chain-contain} and Corollary~\ref{chain-ranks}, which culminate in Theorem~\ref{chain-face-lattice}, shed light on a way to view these tubings, and thus the faces of the stellohedron, as certain chains in the Boolean lattice. Furthermore, in Conjecture~\ref{mn-chain-face} we describe what we think happens for $\mathcal{P}(m,n)$, where $m\neq n$. But first, we review the following result that relates $\mathcal{P}(m,m)$ to the stellohedron; this can be found, in other language, in \cite{Manneville-Pilaud}. See also \cite{RookMonoid}, which gives connections to representation theory.

\begin{theorem}[\protect{\cite[Proposition 56]{Manneville-Pilaud}}]
The polytope $\mathcal{P}(m,m)$ is a realization of $\text{Assoc}(K_{1,m})$.
\end{theorem}

We describe the explicit map for vertices in the remark below.

\begin{remark}
\label{remark:tubes}
The map which sends maximal tubings of $K_{1,m}$ to the vertices of $\mathcal{P}(m,m)$ is as follows. Let $T$ be a maximal tubing of $K_{1,m}$, and for each outer vertex $x_i$, let $t_i$ be the smallest tube containing $x_i$. Then the coordinate in $\mathbb{R}^m$ corresponding to $T$ is $\left(|t_1|-1, |t_2|-1,\ldots,|t_m|-1\right)$. Note that two tubes of the star graph are compatible only if they each contain a single outer vertex and do not contain $*$, or one is contained in the other. So a maximal tubing will have $r$ tubes which are singleton outer vertices and nested tubes of each size from $r+1$ to $m+1$. Moreover, the tube of size $r+1$ must contain each of the $r$ singleton outer vertices along with the inner vertex. Thus such a tubing gets mapped to a coordinate in $\mathbb{R}^m$ with $r$ zeros and whose nonzero entries are $\{m, m-1, \ldots, r+1\}$, which is a vertex of $\mathcal{P}(m,m)$.
\end{remark}

One can view a tubing instead as its corresponding spine, defined below. This will help in our goal of describing a bijection between tubings of the star graph and chains in the Boolean lattice.

\begin{definition} \label{spine} Let $T$ be a tubing of the star graph. The \emph{spine} of $T$ is the poset of tubes of $T$ ordered by inclusion, whose elements are labeled not by the tubes themselves but by the set of new vertices in each tube. For simplicity, we will use the label $i$ in place of $x_i$.
\end{definition}
Spines are defined (in more generality) in \cite[Remark 10]{Manneville-Pilaud} and are called $B$-trees in \cite[Definition 7.7]{Postnikov2}. See Figure~\ref{tubing_spine_chain} for examples of tubings with their corresponding spines, as well as their corresponding chains from the bijection in the following lemma. The \emph{Boolean lattice} $\mathcal{B}_m$ is the poset of all subsets of $\{1,\ldots, m\}$, ordered by inclusion.

\begin{lemma}\label{tubings-chains-bijection}
Tubings of $K_{1,m}$ are in bijection with chains in the Boolean lattice $\mathcal{B}_m$.
\end{lemma}

\begin{proof} 
Given a spine $S$ of a tubing $T$ of $K_{1,m}$, we can construct the corresponding chain in the Boolean lattice as follows. The bottom element of the chain is the subset including anything that is grouped with $*$ in $S$. Each subsequent subset is made by adding in the elements in the next level of $S$, until we reach the top level. As mentioned in Remark~\ref{remark:tubes}, once we reach the first tube containing $*$, we have nested tubes. So the subsets are nested, resulting in a chain in $\mathcal{B}_m$. Any elements not used in the subsets of the chain will be those that appear below the $*$ in $S$. 

Starting with a chain $C \in \mathcal{B}_m$, we can obtain the corresponding spine $S$ (and thus the tubing) by reversing this process. Any elements not in the maximal subset of $C$ will be in the bottom level of $S$ as singletons. Any elements in the minimal chain of $C$ will appear with $*$ in $S$. The new elements that appear in each subsequent subset in $C$ appear together as a new level in $S$. Once we have $S$, we can, of course, recover $T$.
\end{proof}

\begin{figure}[hbtp]
\centering
$\begin{array}{ccccl}
\includegraphics[scale=0.45, valign=c]{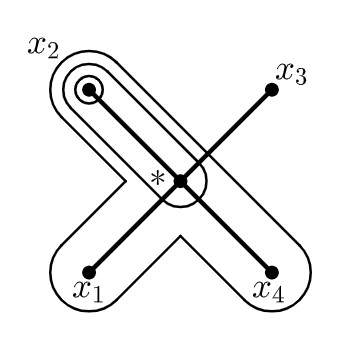} &\longleftrightarrow &\includegraphics[scale=0.45, valign=c]{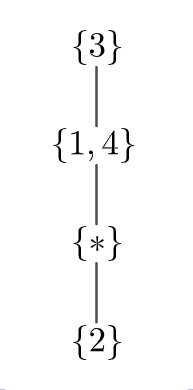} &\longleftrightarrow &\emptyset \subset \{1,4\} \subset \{1,3,4\}\\
\includegraphics[scale=0.45, valign=c]{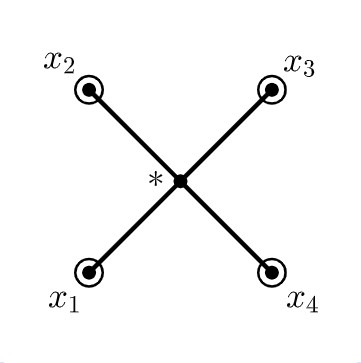} &\longleftrightarrow &\includegraphics[scale=0.45, valign=c]{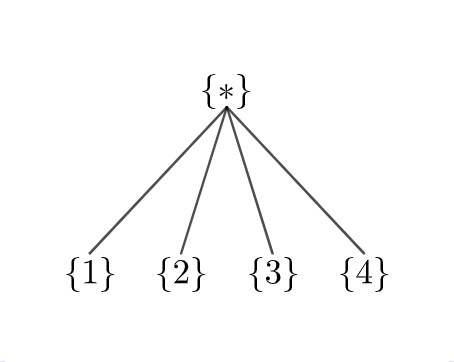} &\longleftrightarrow &\emptyset\\
\includegraphics[scale=0.45, valign=c]{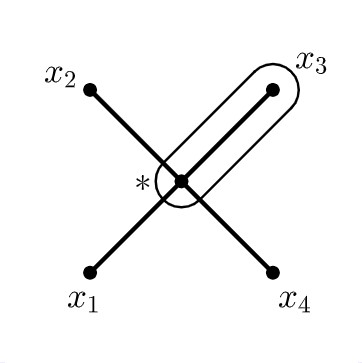} &\longleftrightarrow &\includegraphics[scale=0.45, valign=c]{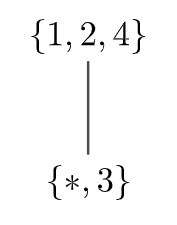} &\longleftrightarrow &\{3\} \subset \{1,2,3,4\}
\end{array}$
\caption{Examples of tubings of $K_{1,4}$ along with their corresponding spines (see Definition~\ref{spine}) and chains in $\mathcal{B}_4$ (via the bijection in Lemma~\ref{tubings-chains-bijection})}
\label{tubing_spine_chain}
\end{figure}

\begin{lemma}\label{chain-contain}
Let $T$ be a $k$-tubing and $T'$ be a $(k+j)$-tubing of $K_{1,m}$, and let $C$ and $C'$ be their corresponding chains in $\mathcal{B}_m$ via the bijection in Lemma~\ref{tubings-chains-bijection}. Then $T \subset T'$ if and only if $C'$ can be obtained from $C$ by $j$ iterations of the following:
\begin{enumerate}
\item adding a non-maximal subset, or
\item removing the same element from every subset.
\end{enumerate}
\end{lemma}

\begin{proof}
Consider $T \subset T'$, \emph{i.e.}~$T'$ is obtained from $T$ by adding tubes. Suppose $T$ and $T'$ differ by adding a single tube, that is, $T = \{t_1, t_2, \ldots, t_k\}$ and $T' = \{t_1, t_2, \ldots, t_k, t'\}$. Let $S$ and $S'$ be their corresponding spines, and let $C$ and $C'$ be their corresponding chains. First note that by the nature of the star graph, a tube either is a singleton outer vertex, $x_i$, or contains the inner vertex,~$*$. Note that a singleton $x_i$ and the singleton $*$ cannot coexist as tubes in a tubing since they are not compatible (they are adjacent).

First consider the case that $t'$ was a singleton outer vertex, $x_i$. This means that in $S$, $i$ was grouped with $*$, while in $S'$, $\{i\}$ now appears below $*$. On the level of chains, this means that $i$ is removed from all of the subsets in $C$ to obtain $C'$.

Now consider the case that $t'$ was not a singleton outer vertex. Then it necessarily contains~$*$. In this case, $S'$ has a new level which was not present in $S$. In particular, this level contains $*$ (and possibly other labels). A new level containing $*$ corresponds to a non-maximal subset being added on the level of chains. In other words, $C'$ is obtained from $C$ by adding a non-maximal subset.

Now suppose $T$ and $T'$ differ by more than one tube, say $T$ is a $k$-tubing and $T'$ is a $(k+j)$-tubing for some $j$. Then $T'$ is obtained from $T$ by adding one tube at a time, $j$ times, and thus $C'$ is obtained from $C$ by $j$ iterations of (1) and/or (2) above.
\end{proof}

We now give a description of the dimension of a face in terms of its corresponding chain. This description involves \emph{missing ranks}, which we define below.

\begin{definition}
\label{def:missing}
Given a chain $C \in \mathcal{B}_m$, we say a rank $j$ is \emph{missing} from $C$ if there is no subset of size $j$ in $C$ and there is a subset of size greater than $j$ in $C$.
\end{definition}

\begin{corollary}\label{chain-ranks}
A face of $\mathcal{P}(m,m)$ is of dimension $k$ if and only if the corresponding chain has $k$ missing ranks.
\end{corollary}

\begin{proof}
We know that adding a tube reduces the dimension of the corresponding face by one. Also, by Lemma~\ref{chain-contain}, we know that adding a tube corresponds to either adding a non-maximal subset or removing an element from every subset in the corresponding chain. In either case, this reduces the number of missing ranks in the chain by one. So, having $k$ missing ranks in the chain corresponds to having $m-k$ tubes, which by definition of the graph associahedron corresponds to a face being of dimension $k$.
\end{proof}

The theorem below follows directly from the above lemmas and corollary.
\begin{theorem} \label{chain-face-lattice}
The face lattice of $\mathcal{P}(m,m)$ is isomorphic to the lattice of chains in $\mathcal{B}_m$, where $C < C'$ if $C'$ can be obtained from $C$ by iterations of (1) and/or (2) from Lemma~\ref{chain-contain}. A face of $\mathcal{P}(m,m)$ is of dimension $k$ if and only if the corresponding chain has $k$ missing ranks.
\end{theorem}

As chains in the Boolean lattice are generally more familiar objects than tubings of graphs, presenting results in terms of these chains is conceptually helpful. In fact, because of the description of the faces of $\mathcal{P}(m,m)$ in terms of chains, we are able to form the following conjecture for $\mathcal{P}(m,n)$.

\begin{conjecture}\label{mn-chain-face}
Faces of $\mathcal{P}(m,n)$ are in bijection with chains in $\mathcal{B}_m$ whose difference between largest and smallest nonempty subsets is at most $n-1$. A face of $\mathcal{P}(m,n)$ is of dimension $k$ if and only if the corresponding chain has $k$ missing ranks
\end{conjecture}

\begin{remark}
This conjecture has been tested and verified for $m,n \leq 4$ using SageMath.
\end{remark}

\subsection{Projection from partial alternating sign matrix polytopes}
\label{ppermutoprojection}
In this subsection, we show that the partial permutohedron is a projection of both  $\pperm(m,n)$ (in Theorem~\ref{pperm-proj}) and $\pasm(m,n)$ (in Theorem~\ref{pasm-proj}). Recall $\phi_z$ and $\mathcal{P}_z(m,n)$ from Definition~\ref{Pz}.

\begin{theorem}\label{pperm-proj}
The projection of $\pperm(m,n)$ by $\phi_z$ is the polytope $\mathcal{P}_z(m,n)$. That is, \[\phi_z(\pperm(m,n))=\mathcal{P}_z(m,n).\] 
\end{theorem}

\begin{proof}
First we need to show $\mathcal{P}_z(m,n) \subseteq \phi_{z}(\pperm(m,n))$. Suppose $v\in \mathcal{P}_z(m,n)$. We wish to show $v\in \phi_{z}(\pperm(m,n))$. By definition, $v=\sum\lambda_i w_i$ for $\lambda_i \geq 0$ with $\sum\lambda_i =1$,
 where the sum is over all length $m$ words $w_i$ whose entries are in $\left\{0,z_1,z_2,\ldots,z_n\right\}$ and whose nonzero entries are distinct. But $w_i = X_i z$ where $X_i\in P_{m,n}$. So $v=\sum\lambda_i X_i z=(\sum\lambda_i X_i)z$, which proves our claim.

Then we need to show that $\phi_{z}(\pperm(m,n)) \subseteq \mathcal{P}_z(m,n)$. 
Define $\hat{z}$ as $z$ with $m-n$ zeros appended if $m\geq n$ and as the largest $n-m$ components of $z$ if $m<n$. 
Let $X = \left\{x_{ij}\right\}$ be an $m \times n$ partial permutation matrix. Then, by Proposition~\ref{maj_prop}, the proof will be completed by showing $Xz \prec_w \hat{z}$ since the convex hull described will then be $\mathcal{P}_z(m,n)$. So, by Definition~\ref{maj}, we need to show:
\[
\displaystyle \sum_{i=1}^k \left(Xz\right)_{[i]} \leq \sum_{i=1}^k \hat{z}_{[i]}, \mbox{ for } 1 \leq k \leq m.
\]
This is true, since each component of the vector $Xz$ is either $0$ or $z_j$ for some $1\leq j\leq n$, because each column of $X$ has at most one nonzero entry.
\end{proof}

\begin{theorem}\label{pasm-proj}
Let $z$ be a strictly decreasing vector in $\mathbb{R}^n$. Then
$\phi_z(\pasm(m,n))=\mathcal{P}_z(m,n)$. 
\end{theorem}
\begin{proof}
Let $z$ be a strictly decreasing vector in $\mathbb{R}^n$. It follows from Theorem~\ref{pperm-proj} and \\ $\pperm(m,n) \subseteq \pasm(m,n)$ that $\mathcal{P}_z(m,n) \subseteq \phi_{z}(\pasm(m,n))$. Thus it only remains to be shown that $\phi_{z}(\pasm(m,n)) \subseteq \mathcal{P}_z(m,n)$.

As in the previous theorem, define $\hat{z}$ as $z$ with $m-n$ zeros appended if $m\geq n$ and as the largest $n-m$ components of $z$ if $m<n$. 
Let $X = \left\{x_{ij}\right\}$ be an $m \times n$ partial alternating sign matrix. Then, by Proposition~\ref{maj_prop}, the proof will be completed by showing $Xz \prec_w \hat{z}$ since the convex hull described will then be $\mathcal{P}_z(m,n)$. So, by Definition~$\ref{maj}$, we need to show:
\begin{equation}
\displaystyle \sum_{i=1}^k \left(Xz\right)_{[i]} \leq \sum_{i=1}^k \hat{z}_i, \mbox{ for } 1 \leq k \leq m.
\end{equation}
To prove this, we will show that $\sum_{i \in I}(Xz)_i \leq \sum_{i=1}^{|I|} \hat{z}_i$ given any $I \subseteq \{1,\ldots,m\}$, so that, in particular, $\sum_{i=1}^{|I|}(Xz)_{[i]} \leq \sum_{i=1}^{|I|} \hat{z}_i$.

We will need to verify the following:
\begin{equation}\label{toprove1}
\sum_{j=1}^{\ell} \sum_{i \in I} x_{ij} \leq \min(\ell,|I|), \mbox{ for } 1 \leq \ell \leq n
\end{equation}

To prove this, note that
\[\sum_{j=1}^{\ell} \sum_{i \in I} x_{ij} = \sum_{i \in I} \sum_{j=1}^{\ell} x_{ij} \leq |I|\]
since $\displaystyle\sum_{j=1}^{\ell} x_{ij} \leq 1$. But since $\displaystyle\sum_{j=1}^{\ell} x_{ij} \geq 0$ and $\displaystyle\sum_{i=1}^m x_{ij}\in\{0,1\}$, we also have that:
\[\sum_{i \in I}\sum_{j=1}^{\ell} x_{ij} \leq \sum_{i=1}^m\sum_{j=1}^{\ell} x_{ij} = \sum_{j=1}^{\ell}\sum_{i=1}^m x_{ij} \leq {\ell},
\]
proving (\ref{toprove1}).

Now we show $\sum_{i \in I}(Xz)_i \leq \sum_{i=1}^{|I|} \hat{z}_i$.
\begin{align*}
\sum_{i \in I}(Xz)_i & = \sum_{i \in I}\sum_{j=1}^n x_{ij}z_{j} = \sum_{j=1}^nz_j\sum_{i \in I} x_{ij} \qquad \qquad  \qquad \quad \hspace{1ex} \quad \quad \mbox{by definition}\\
& = \sum_{{\ell}=1}^{n-1}(z_{\ell} - z_{{\ell}+1})\sum_{j=1}^{\ell}\sum_{i \in I} x_{ij} + z_n\sum_{j=1}^n \sum_{i \in I} x_{ij} \qquad \qquad \mbox{by algebraic manipulation}\\
 & \leq \sum_{{\ell}=1}^{n-1}(z_{\ell} - z_{{\ell}+1}) \sum_{j=1}^{\ell} \sum_{i \in I} x_{ij} + z_n  \min(n,|I|) \qquad \qquad \mbox{by } (\ref{toprove1}) \\
 & = \sum_{{\ell}=1}^{ \min(n,|I|)-1}(z_{\ell} - z_{{\ell}+1})\sum_{j=1}^{\ell} \sum_{i \in I} x_{ij} + \sum_{{\ell} =  \min(n,|I|)}^{n-1}(z_{\ell} - z_{{\ell}+1})\sum_{j=1}^{\ell} \sum_{i \in I} x_{ij} + z_n  \min(n,|I|) \\
 & \leq \sum_{{\ell}=1}^{ \min(n,|I|)-1}(z_{\ell} - z_{{\ell}+1}){\ell} + \sum_{{\ell} =  \min(n,|I|)}^{n-1}(z_{\ell} - z_{{\ell}+1}) |I| +z_n  \min(n,|I|)
 \end{align*}
 by (\ref{toprove1}) and since  $z_{\ell} \geq z_{{\ell}+1}$. Furthermore, this equals
 \begin{align*}
 &  \sum_{{\ell}=1}^{ \min(n,|I|)} z_{\ell} \qquad\mbox{by telescoping sums,}\\
  & = \sum_{{\ell}=1}^{|I|} \hat{z}_{\ell}, \qquad\mbox{since the last } m-n \mbox{ entries of } \hat{z} \mbox{ are zero in the case } n<m.
\end{align*}

Thus $Xz \prec_w \hat{z}$ and so $Xz$ is contained in the convex hull of the partial permutations of $z$. Therefore $\phi_z(\pasm(m,n)) = \mathcal{P}_z(m,n)$.
\end{proof}

\subsection{Volume}
\label{ppermutovolume}

Regarding the volume of $\mathcal{P}(m,n)$, we have the following theorem for $m=2$ and conjecture for $n=2$. We also give normalized volume computations for $m,n \leq 7$ in Figure~\ref{ppermuto_volume}.

\begin{theorem}
\label{ppermvol1}
The polytope $\mathcal{P}(2,n)$ has normalized volume equal to $2n^2-1$.
\end{theorem}

\begin{proof}
$\mathcal{P}(2,n)$ is a 2-dimensional polytope whose extreme points consist of exactly $(0,0)$, $(n,0)$, $(0,n)$, $(n,n-1)$, and $(n-1,n)$. This forms an $n \times n$ square with one corner ``cut off'' by the line segment connecting $(n,n-1)$ to $(n-1,n)$. We can explicitly calculate the area of this region to be $n^2 - \frac{1}{2}$. To obtain the normalized volume we multiply by $\text{dim}\left(\mathcal{P}(2,n)\right)! = 2!$ giving us $2n^2-1$. 
\end{proof}

Refer to Figure~\ref{ppermuto_plot} for the case $m=n=2$.

\begin{figure}[hbtp]
\centering
\begin{tabular}{|c|c|c|c|c|c|c|c|}
\hline
\diaghead(1,-1)%
   {\theadfont nnn}%
   {$m$}{$n$} & 1 & 2    & 3      & 4       & 5        & 6         & 7        \\ \hline
    1     & 1 & 2    & 3      & 4       & 5        & 6         & 7        \\ \hline
    2     & 1 & 7    & 17     & 31      & 49       & 71        & 97       \\ \hline
    3     & 1 & 24   & 129    & 342     & 699      & 1236      & 1989     \\ \hline
    4     & 1 & 77   & 954    & 4554    & 12666    & 27882     & 53370    \\ \hline
    5     & 1 & 238  & 6521   & 59040   & 262410   & 751380    & 1741950  \\ \hline
    6     & 1 & 723  & 42207  & 707669  & 5295150  & 22406130  & 65379150 \\ \hline
    7     & 1 & 2180 & 264501 & 7975502 & 99170254 & 651354480 & 2657217150 \\ \hline
\end{tabular}
\caption{Some normalized volume computations for $\mathcal{P}(m,n)$}
\label{ppermuto_volume}
\end{figure}

\begin{figure}[hbtp]
\includegraphics[scale=0.5]{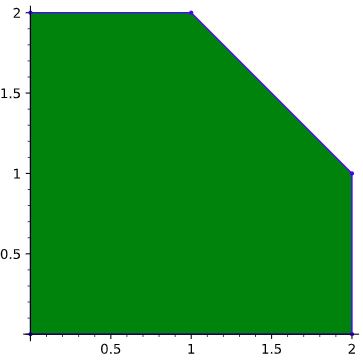} \hspace{0.5cm} \includegraphics[scale=0.4]{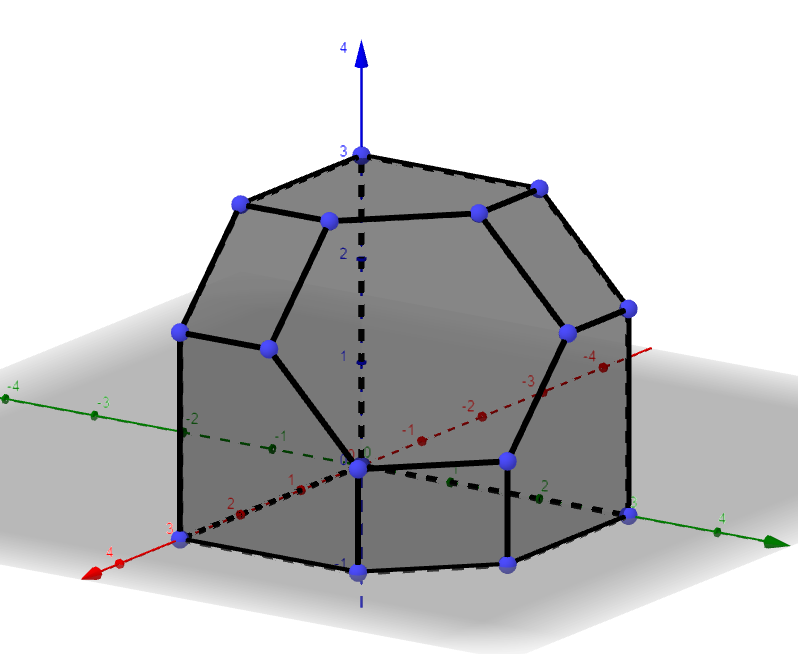}
\caption{Plots of $\mathcal{P}(2,2)$ (left) and $\mathcal{P}(3,3)$ (right).}
\label{ppermuto_plot}
\end{figure}

\begin{conjecture}
\label{ppermvol2}
The polytope $\mathcal{P}(m,2)$ has normalized volume equal to $3^m-m$.
\end{conjecture}

Using SageMath, we have confirmed this conjecture for $m \leq 50$. 

\begin{remark}
We have used SageMath to compute the Ehrhart polynomials for $\mathcal{P}(m,n)$ for $m,n \leq 7$ and note that in all of these cases their coefficients are positive.
\end{remark}

\section*{Acknowledgments}
The authors thank anonymous referees for helpful comments and for the proof of Theorem~\ref{conj:volPPerm}. They thank the developers of \verb|SageMath|~\cite{sage} software, especially the code related to polytopes, which was helpful in our research, and the developers of \verb|CoCalc|~\cite{SMC} for making \verb|SageMath| more accessible. They also thank the OEIS Foundation~\cite{oeis1} and the contributors to the OEIS database for creating and maintaining this resource. JS was supported by a grant from the Simons Foundation/SFARI (527204, JS).

\bibliographystyle{plain}
\bibliography{biblio}

\begin{thebibliography}{10}

\bibitem{oeis1}
OEIS Foundation~Inc. (2020).
\newblock The on-line encyclopedia of integer sequences.
\newblock http://oeis.org/.

\bibitem{AaronAllen}
A.~Allen.
\newblock The combinatorial geometry of rook polytopes, (Undergraduate honors
  thesis, 2017).
\newblock University of Colorado,
  \url{https://scholar.colorado.edu/concern/undergraduate_honors_theses/tq57nr38h}.

\bibitem{BalinskiRussakoff}
M.~L. Balinski and A.~Russakoff.
\newblock On the assignment polytope.
\newblock {\em SIAM Rev.}, 16:516--525, 1974.

\bibitem{behrend}
R.~Behrend and V.~Knight.
\newblock Higher spin alternating sign matrices.
\newblock {\em Electron. J. Combin.}, 14(1):Research Paper 83, 38, 2007.

\bibitem{Osculating}
Roger~E. Behrend.
\newblock Osculating paths and oscillating tableaux.
\newblock {\em Electron. J. Combin.}, 15(1):Research Paper 7, 60, 2008.

\bibitem{BilleraSarangarajan1996}
Louis~J. Billera and A.~Sarangarajan.
\newblock All {$0$}-{$1$} polytopes are traveling salesman polytopes.
\newblock {\em Combinatorica}, 16(2):175--188, 1996.

\bibitem{birkhoff}
G.~Birkhoff.
\newblock Three observations on linear algebra.
\newblock {\em Univ. Nac. Tucum\'an. Revista A.}, 5:147--151, 1946.

\bibitem{Brualdi-Ryser}
R.~Brualdi and H.~Ryser.
\newblock {\em Combinatorial matrix theory}, volume~39 of {\em Encyclopedia of
  Mathematics and its Applications}.
\newblock Cambridge University Press, Cambridge, 1991.

\bibitem{BrualdiBook2006}
Richard~A. Brualdi.
\newblock {\em Combinatorial matrix classes}, volume 108 of {\em Encyclopedia
  of Mathematics and its Applications}.
\newblock Cambridge University Press, Cambridge, 2006.

\bibitem{BrualdiGibson1976}
Richard~A. Brualdi and Peter~M. Gibson.
\newblock The assignment polytope.
\newblock {\em Math. Programming}, 11(1):97--101, 1976.

\bibitem{Cao-et-all}
L.~Cao, S.~Koyuncu, and T.~Parmer.
\newblock A minimal completion of doubly substochastic matrix.
\newblock {\em Linear Multilinear Algebra}, 64(11):2313--2334, 2016.

\bibitem{Devadoss1}
M.~Carr and S.~Devadoss.
\newblock Coxeter complexes and graph-associahedra.
\newblock {\em Topology Appl.}, 153(12):2155--2168, 2006.

\bibitem{Chvatal}
V.~Chv\'{a}tal.
\newblock On certain polytopes associated with graphs.
\newblock {\em J. Combinatorial Theory Ser. B}, 18:138--154, 1975.

\bibitem{Devadoss2}
S.~Devadoss.
\newblock A realization of graph associahedra.
\newblock {\em Discrete Math.}, 309(1):271--276, 2009.

\bibitem{Fortin}
M.~Fortin.
\newblock The {M}ac{N}eille completion of the poset of partial injective
  functions.
\newblock {\em Electron. J. Combin.}, 15(1):Research paper 62, 30, 2008.

\bibitem{RookMonoid}
Jo\"{e}l Gay and Florent Hivert.
\newblock The 0-rook monoid and its representation theory.
\newblock {\em S\'{e}m. Lothar. Combin.}, 78B:Art. 18, 12, 2017.

\bibitem{Heuer}
D.~Heuer.
\newblock {\em On partial permutation and alternating sign matrices: bijections
  and polytopes}.
\newblock PhD thesis, North Dakota State University, Fargo, North Dakota, 2021.

\bibitem{SMC}
SageMath Inc.
\newblock {\em CoCalc Collaborative Computation Online}, 2020.
\newblock {\tt https://cocalc.com/}.

\bibitem{KohlOlsenSanyal}
Florian Kohl, McCabe Olsen, and Raman Sanyal.
\newblock Unconditional reflexive polytopes.
\newblock {\em Discrete Comput. Geom.}, 64(2):427--452, 2020.

\bibitem{TREILLIS}
A.~Lascoux and M.~Sch\"{u}tzenberger.
\newblock Treillis et bases des groupes de {C}oxeter.
\newblock {\em Electron. J. Combin.}, 3(2):Research paper 27, approx. 35, 1996.

\bibitem{Manneville-Pilaud}
T.~Manneville and V.~Pilaud.
\newblock Compatibility fans for graphical nested complexes.
\newblock {\em J. Combin. Theory Ser. A}, 150:36--107, 2017.

\bibitem{majorization}
A.~Marshall and I.~Olkin.
\newblock {\em Inequalities: theory of majorization and its applications},
  volume 143 of {\em Mathematics in Science and Engineering}.
\newblock Academic Press, Inc. [Harcourt Brace Jovanovich, Publishers], New
  York-London, 1979.

\bibitem{Mirsky_sub}
L.~Mirsky.
\newblock On a convex set of matrices.
\newblock {\em Arch. Math.}, 10:88--92, 1959.

\bibitem{Ouchterlony}
E.~Ouchterlony.
\newblock {\em On Young Tableau Involutions and Patterns in Permutations}.
\newblock PhD thesis, Linköpings universitet, 2005.

\bibitem{Postnikov2}
A.~Postnikov.
\newblock Permutohedra, associahedra, and beyond.
\newblock {\em Int. Math. Res. Not. IMRN}, (6):1026--1106, 2009.

\bibitem{ProppManyFaces}
J.~Propp.
\newblock The many faces of alternating-sign matrices.
\newblock In {\em Discrete models: combinatorics, computation, and geometry
  ({P}aris, 2001)}, Discrete Math. Theor. Comput. Sci. Proc., AA, pages
  043--058. Maison Inform. Math. Discr\`et. (MIMD), Paris, 2001.

\bibitem{Schrijver}
A.~Schrijver.
\newblock {\em Combinatorial optimization. {P}olyhedra and efficiency. {V}ol.
  {A}}, volume~24 of {\em Algorithms and Combinatorics}.
\newblock Springer-Verlag, Berlin, 2003.
\newblock Paths, flows, matchings, Chapters 1--38.

\bibitem{SolhjemStriker}
S.~Solhjem and J.~Striker.
\newblock Sign matrix polytopes from {Y}oung tableaux.
\newblock {\em Linear Algebra Appl.}, 574:84--122, 2019.

\bibitem{Stanley1973}
Richard~P. Stanley.
\newblock Linear homogeneous {D}iophantine equations and magic labelings of
  graphs.
\newblock {\em Duke Math. J.}, 40:607--632, 1973.

\bibitem{sage}
W.\thinspace{}A. Stein et~al.
\newblock {\em {S}age {M}athematics {S}oftware ({V}ersion 9.2)}.
\newblock The Sage Development Team, 2020.
\newblock \url{http://www.sagemath.org}.

\bibitem{striker}
J.~Striker.
\newblock The alternating sign matrix polytope.
\newblock {\em Electron. J. Combin.}, 16(1):Research Paper 41, 15, 2009.

\bibitem{vonneumann}
J.~von Neumann.
\newblock A certain zero-sum two-person game equivalent to the optimal
  assignment problem.
\newblock In {\em Contributions to the theory of games, {V}ol. 2}, Annals of
  Mathematics Studies, no. 28, pages 5--12. Princeton University Press,
  Princeton, N. J., 1953.

\end{thebibliography}
\end{document}